\documentclass[hidelinks,onefignum,onetabnum]{siamart220329}

\usepackage{amsfonts}
\usepackage{enumitem}
\usepackage{booktabs}

\newcommand{\bmat}[1]{\begin{bmatrix}#1\end{bmatrix}}
\newcommand{\smat}[1]{\left[\begin{smallmatrix}#1\end{smallmatrix}\right]}

\newcommand{\tu}[1]{\textup{#1}}
\newcommand{\real}{\mathbb{R}}

\makeatletter
\def\ifemptyarg#1{%
  \if\relax\detokenize{#1}\relax 
    \expandafter\@firstoftwo
  \else
    \expandafter\@secondoftwo
  \fi}
\makeatother
\newcommand{\Zbb}{\bar{\boldsymbol{\zeta}}}
\newcommand{\AB}[1][]{\ifemptyarg{#1}{\ensuremath{\bmat{A~\,B}}}{\ensuremath{\bmat{A_{#1}~\,B_{#1}}}}}
\newcommand{\h}{\frac{1}{2}}

\DeclareMathOperator{\Tp}{Tp}

\newsiamremark{remark}{Remark}
\newtheorem{assumption}{\textsc{Assumption}}
\newsiamthm{claim}{Claim}
\newsiamthm{fact}{Fact}
\newsiamthm{problem}{Problem}

\allowdisplaybreaks
\headers{Data-driven input-to-state stabilization}{H. Chen, A. Bisoffi, and C. De Persis}
\title{Data-driven input-to-state stabilization\thanks{Submitted to arXiv on July 08, 2024.}
\funding{This publication is part of the project Digital Twin with project number
P18-03 of the research programme TTW Perspective which is (partly)
financed by the Dutch Research Council (NWO).}}
\author{Hailong Chen\thanks{Engineering and Technology Institute, University of Groningen, 9747 AG Groningen, The Netherlands (\email{hailong.chen@rug.nl}, \email{c.de.persis@rug.nl}).}
\and Andrea Bisoffi\thanks{Department of Electronics, Information,
Bioengineering, Politecnico di Milano, 20133 Milan, Italy
(\email{andrea.bisoffi@polimi.it}).}
\and Claudio De Persis\footnotemark[2]}

\ifpdf
\hypersetup{
  pdftitle={Data-driven input-to-state stabilization},
  pdfauthor={H. Chen, A. Bisoffi and C. De Persis}
}
\fi

\begin{document}

\maketitle

\begin{abstract}
For the class of nonlinear input-affine systems with polynomial dynamics, we consider the problem of designing an input-to-state stabilizing controller with respect to typical exogenous signals in a feedback control system, such as actuator and process disturbances.
We address this problem in a data-based setting when we cannot avail ourselves of the dynamics of the actual system, but only of data generated by it under unknown bounded noise.
For all dynamics consistent with data, we derive sum-of-squares programs to design an input-to-state stabilizing controller, an input-to-state Lyapunov function and the corresponding comparison functions.
This numerical design for input-to-state stabilization seems to be relevant not only in the considered data-based setting, but also in a model-based setting.
Illustration of feasibility of the provided sum-of-squares programs is provided on a numerical example.
\end{abstract}



\section{Introduction}

Input-to-state stability (ISS), originated in \cite{sontag1989smooth}, is a pivotal notion in the analysis and design of nonlinear control systems as it characterizes how stability properties gracefully degrade depending on the ``size'' of exogenous inputs such as actuator, process and sensor disturbances \cite{sontag2008input,Mironchenko2023input}.
Notable features of input-to-state stability are, among others, that it generalizes the classical notion of global asymptotic stability to systems with inputs and is equivalent to the existence of an input-to-state Lyapunov function based on converse theorems \cite{sontag1995characterizations}.
Since input-to-state stability applies to ``open systems'', it is also related to other fundamental properties of nonlinear systems like dissipativity, passivity and $L_2$-gain.

Designing control laws that make a generic nonlinear system input-to-state stable with respect to exogenous inputs remains nontrivial even when a model of the nonlinear system is available.
In \cite[Theorem~3]{liberzon2002universal}, this is done with the ``universal formula'' approach starting from a control Lyapunov function.
In \cite[\S VI-VII]{krstic1998inverse}, ISS control Lyapunov functions are obtained via backstepping.

Designing a control law to ensure input-to-state stabilization of the closed loop or finding an ISS control Lyapunov function becomes even less trivial if one does not have a model of the nonlinear system but relies only on noisy data from such system.
Indeed, the noise in these data intrinsically leads to a set of possible models in the sense of set-membership identification \cite{milanese2004set} and, given this uncertainty, one would like to design a control law that ensures input-to-state stabilization of all closed-loop systems corresponding to this set of possible models, among which the model of the actual system is not distinguishable.

In previous works \cite{guo2021data,bisoffi2022data,luppi2023data} we have noted that input-affine nonlinear systems with polynomial dynamics lend themselves to the design of a control law to achieve asymptotic stabilization or robust invariance thanks to sum-of-squares (SOS) programs \cite{parrilo2003semidefinite,chesi2010lmi,jarvis2005control} and the resulting SOS programs are convex or biconvex \cite[Def.~1.3]{gorski2007biconvex}.
The considered class of nonlinear systems is relevant in that polynomial vector fields can approximate smooth vector fields tightly on compact sets.
Despite limitations \cite{ahmadi2011globally}, certifying dynamical properties of a system through SOS conditions offers a numerically constructive solution to a challenging problem.

In this paper we ask the question whether for this class of nonlinear systems it is possible to design a control law to achieve input-to-state stabilization of the closed-loop system.
We provide a positive answer to this question.
As a by-product, our results for the considered setting seems to be new also for the model-based case in providing a numerical construction of controller and ISS-Lyapunov function.

\subsection*{Contribution}

After characterizing the set of nonlinear input-affine systems with polynomial dynamics that are consistent with the noisy measured data points, our main results provide SOS programs that return a polynomial control law for input-to-state stabilization of the closed-loop system with respect to actuator and process disturbances for all dynamics consistent with data, along with an ISS-Lyapunov function and comparison functions certifying input-to-state stabilization.
In the first two results (Theorems~\ref{thm:biconvex_ISS_w} and \ref{thm:biconvex_ISS_d} for actuator and process disturbances), the SOS programs contain products between decision variables and, to circumvent this nonconvexity, a widespread strategy already used in a model-based case \cite{jarvis2005control,majumdar2013control} is to solve alternately for a subset of decision variables while fixing all other ones.
This motivates us to derive the second two results (Theorems~\ref{thm:convex_ISS_w} and \ref{thm:convex_ISS_d}), which use instead a specific parametrization of controllers and ISS-Lyapunov functions to propose natively convex SOS programs for input-to-state stabilization.
Advantages and disadvantages of these two approaches are illustrated on the numerical example and discussed in Sections~\ref{sec:exper}-\ref{sec:dis}.
The specialization of our results to the model-based case is explored in Section~\ref{sec:modelbased}.

\subsection*{Related work}

We presented a preliminary result in the conference paper \cite{chen2023data}, where we obtained a biconvex SOS program for input-to-state stabilization with respect to measurement noise, spurred by its use in event-triggered control \cite{tabuada2007event}.
For the model-based case, we refer the reader to \cite[\S 5.11]{Mironchenko2023input} for a discussion and results besides the mentioned \cite{liberzon2002universal,krstic1998inverse}.
As said before, these results are based on a control Lyapunov function, which is generally not available and needs to be constructed.
In~\cite{grune2023examples}, a numerical approach to construct a control Lyapunov function  uses deep neural networks to approximate separable control Lyapunov functions.
A sum-of-squares approach to verify input-to-state stability of a given nonlinear system is in~\cite{ichihara2012sum}, unlike which we propose biconvex and convex controller design conditions using noisy data.
For the data-based case, a current depiction of research in data-driven control design of nonlinear systems is in the surveys \cite{de2023learning} and \cite{martin2023guarantees}.
Among works focused on learning dynamical systems properties from data, the goal of \cite{ahmadi2023safely} is learning unknown linear and nonlinear autonomous dynamical systems without violating safety constraints on the state of the system, at the same time.
The goal of~\cite{lavaei2023data} is to certify input-to-state stability of an interconnected system via data-based ISS-Lyapunov functions of its subsystems, where the inputs of each of these are states of other subsystems.

\subsection*{Structure} 
We introduce relevant background knowledge in Section~\ref{sec:pre}. 
Section~\ref{sec:problem} formulates the considered problem and states some motivation. 
Our main results for data-driven input-to-state stabilization, including biconvex and convex programs, are in Section~\ref{sec:main}.
Their effectiveness is verified numerically in Section~\ref{sec:exper} and an overall discussion on the so-obtained results is in Section~\ref{sec:dis}.
Conclusions and future research directions are in Section~\ref{sec:conclu}.

\section{Preliminaries}\label{sec:pre}
\subsection{Notation}\label{sec:notation}

In the sequel we use sum-of-squares (SOS) polynomials and SOS polynomial matrices and we generally write $s(x,y)$ if $s$ is an SOS polynomial in the variables $x$ and $y$ and $S(x,y)$ if $S$ is an SOS polynomial matrix in the variables $x$ and $y$.
We refer to, e.g., \cite{jarvis2005control}, \cite{parrilo2003semidefinite}, \cite{chesi2010lmi}, for excellent surveys on these notions and the related results.
We write $A \succ 0$ or $A \succeq 0$ if a symmetric matrix $A$ is positive definite or positive semidefinite.
We then write $A \succ B$ or $A \succeq B$ if $A-B \succ 0$ or $A-B \succeq 0$.
For a positive semidefinite matrix $A$, $A^{\frac{1}{2}}$ is its unique positive semidefinite square root \cite[p.~440]{horn2013matrix}.
The set of nonnegative real numbers is $\real_{\ge 0}$.
The $n$-dimensional Euclidean space is $\mathbb{R}^n$. 
The Euclidean norm of a vector $x \in \mathbb{R}^n$ is $|x|$. 
The identity matrix is denoted as $I$ or $I_n$, where the positive integer $n$ is the dimension of $I_n$.
A zero matrix is written as $0$.
The maximum and minimum eigenvalue of a symmetric matrix $A$ are denoted as $\lambda_\text{min}(A)$ and $\lambda_\text{max}(A)$, respectively. 
For a matrix $A$, $\|A\|$ is its induced $2$-norm.
We abbreviate the symmetric matrix $\smat{A & B^{\top } \\ B & C}$ to $\smat{A & \star \\ B&C}$ or $\smat{A & B^{\top } \\ \star & C}$.
For a square matrix $A$, $\Tp (A) := A+A^\top$.
For vectors $x_1 \in \real^{n_1}$, \dots, $x_q \in \real^{n_q}$, $(x_1, \dots, x_q) := \bmat{x_1^\top & \dots & x_q^\top}^\top$.
A function $f \colon \real^n \to \real$ is positive definite \cite[p.~5]{isidori1999nonlinear} if it vanishes at zero and is positive elsewhere.
The gradient of a continuously differentiable function $V \colon \real^n \to \real$ at $x$ is $\nabla V(x)$.
For $v_1$, $v_2 \in \real^n$, $\langle v_1, v_2 \rangle := v_1^\top v_2$. 

\subsection{Auxiliary results}

We present the essential notions of input-to-state stability we use in the sequel following \cite[\S 10.4]{isidori1999nonlinear}.
We start recalling the notions of comparison functions.

\begin{definition}[{\cite[Defs.~10.1.1-10.1.2]{isidori1999nonlinear}}]
\label{def:comparison_functions}
A continuous function $\alpha\colon [0,a) \to [0, +\infty)$ is said to belong to class $\mathcal{K}$ if it is strictly increasing and $\alpha(0) = 0$.
If $a = +\infty$ and $\lim_{r \to + \infty} \alpha(r) = + \infty$, the function is said to belong to class $\mathcal{K}_\infty$.

A continuous function $\beta \colon [0,a) \times [0,+\infty) \to [0,+\infty)$ is said to belong to class $\mathcal{KL}$ if, for each fixed $s$, the function $r \mapsto \beta(r,s)$ (from $[0,a)$ to $[0,+\infty)$) belongs to class $\mathcal{K}$ and, for each fixed $r$, the function $s \mapsto \beta(r,s)$ (from $[0,+\infty)$ to $[0,+\infty)$) is decreasing and approaches $0$ as $s$ approaches $+\infty$.
\end{definition}

Consider a nonlinear system
\begin{align}
\label{nonlinsysisidori}
\dot{x} = F(x,v)
\end{align}
with state $x \in \real^n$, input $v \in \real^p$, $F(0,0) = 0$ and $F(\cdot, \cdot)$ locally Lipschitz on $\real^n \times \real^p$.
The input function $v \colon [0,+\infty) \to \real^p$ of \eqref{nonlinsysisidori} can be any piecewise continuous bounded function.
The set of all such functions, endowed with the supremum norm
\begin{align*}
\| v(\cdot) \|_\infty := \sup_{t\ge 0} |v(t)|
\end{align*}
is denoted by $L_\infty$.
For this system, input-to-state stability (ISS) is defined next.

\begin{definition}[{\cite[Def.~10.4.1]{isidori1999nonlinear}}]
\label{def:ISS}
System~\eqref{nonlinsysisidori} is said to be input-to-state stable if there exist a class $\mathcal{KL}$ function $\beta(\cdot,\cdot)$ and a class $\mathcal{K}$ function $\gamma(\cdot)$
such that for each $x^\circ \in \real^n$ and each input $v(\cdot) \in L_\infty$, the solution $x(\cdot)$ to~\eqref{nonlinsysisidori} for initial state $x(0) = x^\circ$ and input $v(\cdot)$ satisfies
\begin{align}
\label{ISS_bound_solutions}
|x(t)| \le \beta(|x^\circ|,t) + \gamma(\| v(\cdot)\|_\infty) \quad \forall t \ge 0.
\end{align}
\end{definition}

We note that if one is interested in having input-to-stability, as we are here, the assumption that \eqref{nonlinsysisidori} satisfies $F(0,0) = 0$ is without loss of generality.

The main tool to characterize input-to-state stability are ISS-Lyapunov functions, as defined next.
\begin{definition}[{\cite[Def.~10.4.2, Lemma 10.4.2]{isidori1999nonlinear}}]
\label{def:ISS_lyap_fun}
A continuously differentiable function $V \colon \real^n \to \real$ is called an ISS-Lyapunov function for~\eqref{nonlinsysisidori} if there exist class $\mathcal{K}_\infty$ functions $\underline{\alpha}(\cdot)$, $\overline{\alpha}(\cdot)$, $\alpha(\cdot)$ and a class $\mathcal{K}$ function $\sigma(\cdot)$ such that
\begin{subequations}
\label{ISS_lyap_fun_prop}
\begin{align}
& \underline{\alpha}(|x|)\le V(x) \le \overline{\alpha}(|x|) \quad \forall x \in \real^n \label{ISS_lyap_fun_prop:bounds} \\
& \langle \nabla V(x), F(x,v) \rangle \le - \alpha(|x|) + \sigma(|v|) \quad \forall x \in \real^n ,\,  v \in \real^p. \label{ISS_lyap_fun_prop:dissip_ineq}
\end{align} 
\end{subequations}
\end{definition}

A key relation between ISS and an ISS-Lyapunov function is recalled next.

\begin{fact}[\cite{sontag1995characterizations}, {\cite[Thm.~10.4.1]{isidori1999nonlinear}}]
\label{fact:equiv_iss}
System~\eqref{nonlinsysisidori} is input-to-state stable if and only if there exists an ISS-Lyapunov function.
\end{fact}

Finally, we report a seminal result on input-to-state stability, originally stated as ``smooth stabilizability implies smooth input-to-state stabilizability''.

\begin{fact}[{\cite[Thm.~1]{sontag1989smooth}}]\label{fact_GAS_ISS}
Consider the system
\begin{align*}
\dot{x} = f(x) + G(x) u =: f(x) + \bmat{g_1(x) & \dots & g_m(x)} u
\end{align*}
with $f$, $g_1$, \dots, $g_m \colon \real^n \to \real^n$ smooth and $f(0) = 0$.
Suppose that there is a smooth map $K \colon \real^n \to \real^m$ with $K(0) = 0$ such that $\dot{x} = f(x) + G(x) K(x)$ is globally asymptotically stable at equilibrium $x=0$ (``smooth stabilizability'').

Then, there is a smooth map $\tilde{K} \colon \real^n \to \real^m$ with $\tilde{K}(0) = 0$ so that $\dot{x} = f(x) + G(x) (\tilde{K}(x) + w)$ is input-to-state stable with respect to $w$ (``smooth input-to-state stabilizability'').
\end{fact}

In the sequel, the next lemma allows us to design polynomial class $\mathcal{K}_\infty$ functions.
\begin{lemma}[{\cite[Lemma~1]{chen2023data}}]\label{lem:classk}
Consider $\alpha \colon \mathbb{R}_{\ge 0} \to \mathbb{R}_{\ge 0}$ defined as
\begin{align*}
\alpha(r):=\sum_{k=1}^N c_k r^{2 k}
\end{align*}
for some integer $N \ge 1$.
If the scalars $c_1, c_2, \dots, c_N$ satisfy $c_1 \geq 0$, $c_2 \geq 0$, \dots, $c_N \geq 0$ and $c_1+c_2+\dots+c_N>0$, then $\alpha$ belongs to class $\mathcal{K}_{\infty}$.
\end{lemma}
\begin{proof}
The function $\alpha$ is continuous and $\alpha(0)=0.$ It is strictly increasing if $\alpha(r_2)>\alpha(r_1)$ for all $r_1$ and $r_2$ with $r_2>r_1 \geq 0$. For two arbitrary $r_2$ and $r_1$ such that $r_2>r_1 \geq 0$, $ \alpha(r_2)-\alpha(r_1)=\sum_{k=1}^N c_k(r_2^{2 k}-r_1^{2 k})$. 
This sum of nonnegative quantities is zero only if $c_1=c_2=\dots=c_N=0$. 
This is excluded by hypothesis and, thus, $\alpha(r_2)>\alpha(r_1)$. 
Since at least one of $c_1$, $\ldots, c_N$ is nonzero as a consequence of the hypothesis, $\lim_{r \to +\infty} \alpha(r) = +\infty$.
\end{proof}

\section{Problem formulation and motivation}\label{sec:problem}

Consider the class of nonlinear input-affine systems with polynomial dynamics
\begin{equation}\label{sys2}
\dot{x}=f_{\star}(x)+g_{\star}(x) u,
\end{equation}
where $f_{\star}$ and $g_{\star}$ have thus polynomials as their elements. 
We consider this class because, in general, polynomial vector fields can approximate smooth vector fields tightly on compact sets and, here, they enable the combination of data-based conditions with powerful design tools provided by sum-of-squares (SOS). 
The actual expressions of $f_{\star}$ and $g_{\star}$ are \emph{unknown} to us.
Nonetheless, we introduce the next assumption on them.
\begin{assumption}
\label{assumpt:sys}
We know function libraries $x \mapsto Z(x)$ and $x \mapsto W(x)$, with $Z \colon \real^n \to \real^N$ and $W \colon \real^n \to \real^{M \times n}$, such that:
$Z(0) = 0$, their elements are monomials of $x$ and,
for some constant coefficient matrices $A_\star \in \real^{n \times N}$ and $B_\star \in \real^{n \times M}$,
\begin{align*}
f_\star(x)=A_\star Z(x) \text{ and } g_\star(x)=B_\star W(x) \quad \forall x \in \real^n.
\end{align*}
\end{assumption}
With Assumption~\ref{assumpt:sys}, \eqref{sys2} rewrites as
\begin{equation}\label{sys_poly_ol}
\dot{x}=A_{\star} Z(x)+B_{\star} W(x) u 
\end{equation}
for the \emph{known} $Z$ and $W$ and some \emph{unknown} coefficient matrices $A_{\star}$ and $B_{\star}$. 

We make the next remark on Assumption~\ref{assumpt:sys} and the selection of function libraries $Z$ and $W$.

\begin{remark}
\label{rmk:regressor}
In principle, a foolproof way of satisfying Assumption~\ref{assumpt:sys} is to let $Z$ and $W$ contain all monomials with degree less than a large positive integer since Assumption~\ref{assumpt:sys} requires that if a monomial is present in $f_{\star}$ (or $g_{\star}$), it must be present in $Z$ (or $W$), but not vice versa. 
This choice, however, can be overly conservative because the redundant terms can lead to an increased computational cost and possible infeasibility of the resulting SOS programs. 
Still, even from a high-level knowledge of the system under study \cite{ahmadi2023learning}, it can be apparent which monomials are in $f_{\star}$ (or $g_{\star}$), and thus in $Z$ (or $W$); otherwise, techniques such as those in \cite{brunton2016discovering} can be preliminarily employed. 
In general, a parsimonious choice of the monomials in $Z$ and $W$ is best suited.
\end{remark}

\subsection{Data-driven representation with noisy data}
\label{subsec:data}

In this subsection, we develop a data-driven representation to compensate for the lack of knowledge on $A_{\star}$ and $B_{\star}$. 
Our approach is to gather information about the system from data and, based only on these data, design a state-feedback controller enforcing input-to-state stability with respect to exogenous inputs. 
Data are collected in an open-loop experiment. 
We consider $T$ data points that are generated by \eqref{sys_poly_ol} in the inevitable presence of a noise term $d$, namely, for $i=0, \dots, T-1$,
\begin{equation}
\label{data_collect}
\dot{x}(t_i)=A_{\star} Z(x(t_i))+B_{\star} W(x(t_i)) u(t_i)+d(t_i) 
\end{equation}
where we measure input, state, and state derivative at times $t_i$, i.e., $u(t_i)$, $x(t_i)$, $\dot{x}(t_i)$, but we know only a norm bound on the noise samples, as in the next assumption.
\begin{assumption}\label{assumpt:noise}
For $\delta>0$, all $d(t_0), \dots, d(t_{T-1})$ belong to
\begin{equation}\label{disturbance_bound}
\mathcal{D}:=\{ d \in \mathbb{R}^n\colon |d|^2 \leq \delta \}.
\end{equation}
\end{assumption}

For simplicity, we assume to measure the state derivative $\dot{x}$ at times $t_0, \ldots, t_{T-1}$. 
When not available, $\dot{x}$ can be recovered from a denser sampling of $x$, e.g., using techniques from continuous-time system identification \cite{garnier2003continuous}: at any rate, these techniques allow reconstructing $\dot{x}$ with some error that we factor in through noise $d$. 
See \cite[Appendix A]{de2023event} for an alternative approach that does not require measuring $\dot{x}$, but leads anyhow to the sets $\mathcal{I}_i$ and $\mathcal{I}$ defined below in \eqref{uni_dataset} and \eqref{cap_dataset}. 
Moreover, we emphasize that sampling needs \emph{not} to be uniform (i.e., we do not need $t_{T-1}- t_{T-2}=\dots=t_2-t_1=t_1-t_0$) and, as a matter of fact, data points $\{u(t_i), x(t_i), \dot{x}(t_i)\}_{i=0}^{T-1}$ are \emph{not} required to be collected from a single trajectory but can arise from multiple trajectories, which is especially useful when collecting data from a system with diverging solutions.

Based on the collected data, we can characterize the set of matrices $\AB$ consistent with the $i$-th data point $\{u(t_i), x(t_i), \dot{x}(t_i)\}$, $i = 0, \dots, T-1$, and the instantaneous bound $\mathcal{D}$ in~\eqref{disturbance_bound} as
\begin{equation}\label{uni_dataset}
\mathcal{I}_i :=\Big\{\AB\colon
\dot{x}(t_i)=\AB
\smat{
Z(x(t_i)) \\
W(x(t_i)) u(t_i)}
+d,|d|^2 \le \delta \Big\},
\end{equation}
namely, the set of all matrices $\AB$ that could have generated the data point $\{ u(t_i), x(t_i), \dot{x}(t_i) \}$ for some $d$ complying with bound $\mathcal{D}$, cf.~\eqref{data_collect}. 
The set of matrices consistent with all data points and the instantaneous bound $\mathcal{D}$ is then
\begin{equation}\label{cap_dataset}
\mathcal{I}:=\bigcap_{i=0}^{T-1} \mathcal{I}_i.    
\end{equation}
We emphasize that $\AB[\star] \in \mathcal{I}$ since $d(t_0) \in \mathcal{D}$, \dots, $d(t_{T-1}) \in \mathcal{D}$.

\subsection{Overapproximation of the set $\mathcal{I}$ of consistent matrices}

To construct a more tractable ellipsoidal overapproximation of the set $\mathcal{I}$, we rely on \cite{bisoffi2021trade,luppi2023data}, which extend an approach for classical ellipsoids \cite[Section 3.7.2]{boyd1994linear} to the matrix ellipsoids appearing in the data-based problems under consideration.
For data point $i=0, \cdots, T-1$, define
\begin{equation}
\label{sol_overapp}
\begin{aligned}
& \boldsymbol{C}_i:=\dot{x}\left(t_i\right) \dot{x}\left(t_i\right)^{\top }-\delta I, \quad
\boldsymbol{B}_i:=
-\bmat{
Z(x(t_i)) \\
W(x(t_i)) u(t_i)
}
\dot{x}(t_i)^\top, \\
& \boldsymbol{A}_i:=
\bmat{
Z(x(t_i)) \\
W(x(t_i)) u(t_i)
}
\bmat{
Z(x(t_i)) \\
W(x(t_i)) u(t_i)
}^\top .
\end{aligned}
\end{equation}
Consider the set
\begin{align*}
\bar{\mathcal{I}}:=\{\AB=\zeta^\top : \bar{\mathbf{B}}^\top  \bar{\mathbf{A}}^{-1} \bar{\mathbf{B}}+\bar{\mathbf{B}}^\top \zeta+\zeta^\top  \bar{\mathbf{B}}+\zeta^\top  \bar{\mathbf{A}} \zeta \preceq I \}
\end{align*}
where the matrices $\bar{\mathbf{A}}$ and $\bar{\mathbf{B}}$ are designed by solving
\begin{equation}\label{overapp}
\begin{aligned}
& \text {minimize} \quad & &-\log \det \bar{\mathbf{A}} \quad (\text{over } \bar{\mathbf{A}}, \bar{\mathbf{B}}, \tau_0, \dots, \tau_{T-1}) \\
& \text {subject to} \quad& & \bar{\mathbf{A}} \succ 0, \tau_0 \geq 0, \dots, \tau_{T-1} \geq 0, \\
& & &{\begin{bmatrix}
-I-\sum_{i=0}^{T-1} \tau_i \boldsymbol{C}_i & \star & \star \\
\bar{\mathbf{B}}-\sum_{i=0}^{T-1} \tau_i \boldsymbol{B}_i & \bar{\mathbf{A}}-\sum_{i=0}^{T-1} \tau_i \boldsymbol{A}_i & \star \\
\bar{\mathbf{B}} & 0 & -\bar{\mathbf{A}}
\end{bmatrix} \preceq 0}.
\end{aligned}
\end{equation}
In \eqref{overapp}, the objective function corresponds to the size of the set $\bar{\mathcal{I}}$ \cite[\S 2.2]{bisoffi2021trade} and the constraints ensure the next result.
\begin{fact}[{\cite[\S 5.1]{bisoffi2021trade}}]\label{fact:overapp_set}
If $\bar{\mathbf{A}}$ and $\bar{\mathbf{B}}$ are a solution to \eqref{overapp}, the set $\bar{\mathcal{I}}$ satisfies $\mathcal{I}\subseteq\bar{\mathcal{I}}$.
\end{fact}
The so-obtained $\bar{\mathcal{I}}$ is thus an overapproximation of $\mathcal{I}$ and we use it since, unlike the set $\mathcal{I}$, it is a matrix ellipsoid and this enables the developments of the sequel. 
We would like to reassure the reader on the feasibility of the optimization program in \eqref{overapp} by recalling the next fact, with definitions
\begin{align*}
Z_0&:= \bmat{Z(x(t_0))\hspace*{25pt} & \dots & Z(x(t_{T-1}))\hspace*{37pt}} \in \mathbb{R}^{N \times T},\\
W_0&:= \bmat{
W(x(t_0))u(t_0) & \dots&  W(x(t_{T-1})) u(t_{T-1})
} \in \mathbb{R}^{M \times T}.
\end{align*}

\begin{fact}[{\cite[Lemma~2]{luppi2023data}}]\label{fact_full_rank}
If the matrix $\smat{Z_0 \\ W_0}$ has full row rank, then the optimization program in~\eqref{overapp} is feasible.
\end{fact}
Intuitively, collecting more data points can ensure that $\left[\begin{smallmatrix}Z_0 \\ W_0\end{smallmatrix}\right]$ becomes full row rank if it is not, since these additional data points would constitute additional columns of $\left[\begin{smallmatrix}Z_0 \\ W_0\end{smallmatrix}\right]$; hence, Fact~\ref{fact_full_rank} suggests that the more data points, the higher the chance that $\left[\begin{smallmatrix}Z_0 \\ W_0\end{smallmatrix}\right]$ has full row rank and, in turn, that \eqref{overapp} is feasible.

If $\bar{\mathbf{A}}$ and $\bar{\mathbf{B}}$ are a solution to \eqref{overapp}, we have $\bar{\mathbf{A}} \succ 0$ by construction. Then, we define
\begin{equation}\label{solution_overapp}
\Zbb:=-\bar{\mathbf{A}}^{-1} \bar{\mathbf{B}} \text { and } \bar{\mathbf{Q}}:=I,
\end{equation}
and rewrite the set $\bar{\mathcal{I}}$ as
\begin{equation}
\label{set_overapproximation}
\bar{\mathcal{I}}=\big\{\AB = \zeta^\top = (\Zbb+\bar{\mathbf{A}}^{-\frac{1}{2}} \Upsilon \bar{\mathbf{Q}}^{\frac{1}{2}} )^\top \colon \|\Upsilon\| \le 1 \big\} .
\end{equation}

\subsection{Motivation}

One of our subsequent results will be data-driven input-to-state stabilization with respect to actuator disturbances.
On the surface, the model-based result in Fact~\ref{fact_GAS_ISS} establishes a relationship between global asymptotic stability (GAS) and input-to-state stability with respect to actuator disturbances.
However, Fact~\ref{fact_GAS_ISS} would not allow for data-driven control design since the controller in Fact~\ref{fact_GAS_ISS} is constructed based on a precise dynamical model of the system, which is not available in our data-driven setting.
The unviability of this approach is our motivation to explore a new design technique for data-driven input-to-state stabilization. 
Let us explain this point in detail.

First, the next theorem provides data-based conditions for global asymptotic stabilization through comparison functions using Lemma~\ref{lem:classk}, starting from the model-based result \cite[p.~440]{sontag1989smooth}\footnote{As argued in \cite[p.~440]{sontag1989smooth}, if the derivative of the Lyapunov function along solutions is negative definite, then it can be upperbounded by the opposite of a suitable class $\mathcal{K}_\infty$ function of the norm of the state.}.

\begin{theorem}
\label{thm:bilinear_GAS}
For data points $\{u(t_i), x(t_i), \dot{x}(t_i)\}_{i=0}^{T-1}$ and under Assumptions~\ref{assumpt:sys} and \ref{assumpt:noise}, let the optimization program \eqref{overapp} be feasible. 
Suppose there exist a scalar $\mu>0$, a polynomial vector $k$ with $k(0)=0$, polynomials $\alpha_1$, \dots, $\alpha_3$, $V$, $\lambda$, SOS polynomials $s_{\alpha_1}$, $s_{\alpha_2}$, $s_\lambda$ and an SOS polynomial matrix $S_\nabla$ such that for all $r$ and $x$,%
\begin{subequations}\label{v_gas_data}%
\begin{align}
&\alpha_i(r)=\sum_{k=1}^{N_i} c_{i k} r^{2 k}, c_{i 1} \geq 0, \dots, c_{i N_i} \geq 0,  \sum_{k=1}^{N_i} c_{i k} \geq \mu  \qquad \text {for } i=1, 2, 3
\label{v_gas_data_a}\\
&V(x)-\alpha_1(|x|)=s_{\alpha_1}(x),\quad \alpha_2(|x|)-V(x)=s_{\alpha_2}(x),\label{v_gas_data_b}\\
&\lambda(x)-\mu=s_\lambda(x),\label{v_gas_data_c}\\
&
\bmat{
\alpha_3(|x|)+\frac{\partial{V}}{\partial{x}}(x)\boldsymbol{\Zbb}^\top
\smat{Z(x)\\W(x)k(x)} &\star&\star\\
\mathbf{\bar{Q}}^\frac{1}{2}\frac{\partial{V}}{\partial{x}}(x)^\top &-2\lambda(x) I&\star\\ \lambda(x) \mathbf{\bar{A}}^{-\frac{1}{2}}\left[\begin{smallmatrix}Z(x)\\W(x)k(x)\end{smallmatrix}\right]&0&-2\lambda(x) I
} =-S_\nabla(x). \label{v_gas_data_d}
\end{align}
\end{subequations}
Then, the origin of the closed-loop system $\dot{x}=AZ(x)+BW(x)k(x)$ 
is globally asymptotically stable for all $\AB \in \mathcal{I}$, and in particular for $\AB[\star]$.
\end{theorem}
\begin{proof}
The proof follows from arguments similar to those in~\cite[Thm. 3]{bisoffi2022data} and in the following Theorems~\ref{thm:biconvex_ISS_w} and \ref{thm:biconvex_ISS_d}, so it is omitted.
\end{proof}

Second, consider 
\begin{equation*}
    \dot{x}=A_\star Z(x)+B_\star W(x)(u+w)
\end{equation*}
where $w\in\mathbb{R}^m$ is an actuator disturbance and suppose we set out to use the controller $k$ returned by Theorem~\ref{thm:bilinear_GAS} to obtain a modified controller $\tilde{k}$ that ensures input-to-state stabilization of
\begin{equation*}
    \dot{x}=A_\star Z(x)+B_\star W(x)(\tilde{k}(x)+w)
\end{equation*}
with respect to $w$, using Fact~\ref{fact_GAS_ISS}.
This would lead to the next result.

\begin{lemma}
Under the hypothesis of Theorem~\ref{thm:bilinear_GAS}, among which \eqref{v_gas_data} being satisfied, then for every $\AB\in\mathcal{I}$ (in particular, $\AB[\star]$), there exists a polynomial controller $\tilde{k}$ such that 
\begin{equation}\label{system_poly_cl}
\dot{x}=AZ(x)+BW(x)(\tilde{k}(x)+w)
\end{equation}
is ISS with respect to $w$.
\end{lemma}
\begin{proof}
The proof follows by borrowing the arguments in \cite[Proof of Theorem~1]{sontag1989smooth}.
Let $V$ and $k$ be the Lyapunov function and controller returned by~\eqref{v_gas_data} and let, for each $\AB\in\bar{\mathcal{I}}$ and $x$,
\begin{align*}
\rho(x) & := -\frac{\partial V}{\partial x}(x) ( AZ(x)+B W(x) k(x) ), 
\quad \tilde{k}(x) := k(x) - \frac{\rho(x)}{2 m} \bigg( \frac{\partial V}{\partial x}(x) B W(x) \bigg)^\top.
\end{align*}
As in \cite[Proof of Theorem~1]{sontag1989smooth}, one can show that $V$ is an ISS-Lyapunov function for the closed loop \eqref{system_poly_cl} and \eqref{system_poly_cl} is input-to-state stable with respect to $w$.
By its definition, $\tilde{k}$ is a polynomial vector.
\end{proof}

The result shows that for any  pair $\AB \in \mathcal{I}$, an input-to-state stabilizing controller (with respect to $w$) and a corresponding ISS-Lyapunov function can be designed. 
Unfortunately, since the $\AB[\star]$ defining the actual dynamics is unknown, the previous result cannot explicitly return a polynomial controller $\tilde k$ that makes the actual closed-loop system input-to-state stable. In fact, what is needed is a controller $\tilde k$ that, for all $\AB \in \mathcal{I}$, ensures input-to-state stabilization of the system corresponding to $\AB$.
This amounts to the design of a feedback controller $k$ and, as per Definition~\ref{def:ISS_lyap_fun}, an ISS-Lyapunov function $V$ and class $\mathcal{K}_{\infty}$ functions $\alpha_1$, $\alpha_2$, $\alpha_3$, $\alpha_4$ so that the conditions in~\eqref{ISS_lyap_fun_prop} hold true for all $\AB \in \mathcal{I}$, based on our priors and the available data. 
This yields the problem statements below.

\begin{problem}\label{problem:actuator_disturbance}
For data points $\{u(t_i), x(t_i), \dot{x}(t_i)\}_{i=0}^{T-1}$ and the resulting set in $\mathcal{I}$ in \eqref{cap_dataset}, design a controller $u=k(x)$, an ISS-Lyapunov function $V$ and class $\mathcal{K}_{\infty}$ functions $\alpha_1$, $\alpha_2, \alpha_3, \alpha_4$ such that for all $\AB \in \mathcal{I}$, $x$ and $w$,
\begin{subequations}
\begin{align}
&\alpha_1(|x|) \leq V(x) \leq \alpha_2(|x|), \label{problem_actuator_a}\\
&\Big\langle\nabla V(x),\AB
\smat{
Z(x) \\
W(x)(k(x)+w)
}
\Big\rangle \le -\alpha_3(|x|)+\alpha_4(|w|) .\label{problem_actuator_b}
\end{align}
\end{subequations}
\end{problem}
\begin{problem}\label{problem:process_noise}
For data points $\{u(t_i), x(t_i), \dot{x}(t_i)\}_{i=0}^{T-1}$ and the resulting set in $\mathcal{I}$ in \eqref{cap_dataset}, design a controller $u=k(x)$, an ISS-Lyapunov function $V$ and class $\mathcal{K}_{\infty}$ functions $\alpha_1$, $\alpha_2, \alpha_3, \alpha_4$ such that for all $\AB \in \mathcal{I}$, $x$ and $d$,
\begin{subequations}
\begin{align}
&\alpha_1(|x|) \leq V(x) \leq \alpha_2(|x|),\label{problem_process_a}\\
&\Big\langle\nabla V(x),\AB
\smat{
Z(x) \\
W(x) k(x)
} + d \Big\rangle \le -\alpha_3(|x|)+\alpha_4(|d|).\label{problem_process_b}
\end{align}
\end{subequations}
\end{problem}

We emphasize that these problems involve finding a solution $V$ to a dissipativity-like inequality as those in \eqref{problem_actuator_b} and \eqref{problem_process_b}, from data. 
Section \ref{sec:main} is devoted to solving these problems.

\section{Main results}\label{sec:main}

In this section, we provide our main results. With matrices $\bar{\mathbf{A}}$, $\Zbb$, and $\bar{\mathbf{Q}}$ obtained from noisy data $\{u(t_i), x(t_i), \dot{x}(t_i)\}_{i=0}^{T-1}$ via \eqref{overapp}, we establish input-to-state stability of all the dynamics consistent with data in $\mathcal{I}$ with respect to exogenous inputs (including actuator and process disturbances) by joint synthesis of ISS-Lyapunov functions and state feedback controllers.

\subsection{ISS with respect to actuator disturbances}

In this subsection, we consider the system
\begin{equation}
\label{system_op_nl_act_dist}
    \dot{x}= A_\star Z(x)+B_\star W(x)(u+w) =: f^{\tu{a}}_{A_\star,B_\star}(x,w,u)
\end{equation}
where $w\in\mathbb{R}^m$ is the actuator disturbance.
In the next theorem, we provide conditions to find a controller $u=k(x)$ that makes the closed loop input-to-state stable with respect to $w$ for all possible dynamics consistent with data, as certified by an ISS-Lyapunov function and comparison functions.

\begin{theorem}
\label{thm:biconvex_ISS_w}
For data points $\{u(t_i), x(t_i), \dot{x}(t_i)\}_{i=0}^{T-1}$ and under Assumptions~\ref{assumpt:sys}-\ref{assumpt:noise}, let the optimization program \eqref{overapp} be feasible.
Suppose there exist a scalar $\mu>0$, a polynomial vector $k$ with $k(0)=0$, polynomials $\alpha_1$, \dots, $\alpha_4$, $V$, $\lambda$, SOS polynomials $s_{\alpha_1}$, $s_{\alpha_2}$, $s_\lambda$ and an SOS polynomial matrix $S_\nabla$ such that for all $r$, $x$ and $w$,
\begin{subequations}
\label{biconvex_ISS_w}
\begin{align}
& \alpha_i(r)=\sum_{k=1}^{N_i} c_{i k} r^{2 k}, c_{i 1} \geq 0, \dots, c_{i N_i} \ge 0, \sum_{k=1}^{N_i} c_{i k} \geq \mu   \qquad\text {for } i=1, \dots, 4,
\label{biconvex_ISS_w:alpha}\\
&V(x)-\alpha_1(\vert x\vert)=s_{\alpha_1}(x),\quad \alpha_2(|x|)-V(x)=s_{\alpha_2}(x),\label{biconvex_ISS_w:bounds_V}\\
&\lambda(x,w)-\mu=s_\lambda(x,w),\label{biconvex_ISS_w:lambda}\\
& \bmat{
\left\{\begin{matrix}
\alpha_3(|x|)-\alpha_4(|w|)\\
+\frac{\partial{V}}{\partial{x}}(x)\boldsymbol{\Zbb}^\top  \smat{Z(x)\\W(x)(k(x)+w)} 
\end{matrix}\right\} & \star & \star\\
\mathbf{\bar{Q}}^\frac{1}{2}\frac{\partial{V}}{\partial{x}}(x)^\top &-2\lambda(x,w) I & \star\\ 
\lambda(x,w) \mathbf{\bar{A}}^{-\frac{1}{2}}
\smat{Z(x)\\W(x)(k(x)+w)} & 0 &-2\lambda(x,w) I 
} =-S_\nabla(x,w). \label{biconvex_ISS_w:dissip}
\end{align}
\end{subequations}
Then, the system
\begin{equation}
\label{theorem_disturbance_sys}
\dot{x}=\AB
\bmat{
Z(x) \\
W(x) (k(x)+w)
}
= f^{\tu{a}}_{A,B}(x,w,k(x))
\end{equation}
is ISS with respect to actuator disturbance $w$ for all $\AB \in \bar{\mathcal{I}}$, and in particular for $\AB[\star]$.
\end{theorem}

\begin{proof}
By \eqref{biconvex_ISS_w:alpha} and Lemma~\ref{lem:classk}, $\alpha_1, \dots, \alpha_4$ are class $\mathcal{K}_\infty$ functions (when their domain is restricted to $\mathbb{R}_{\geq0}$). Since $s_{\alpha_1}$ and $s_{\alpha_2}$ are SOS polynomials, \eqref{biconvex_ISS_w:bounds_V} implies \eqref{problem_actuator_a}.

If we have for $V$, $\alpha_3$, $\alpha_4$ such that for all $x$, $w$ and all $\Upsilon$ with $\|\Upsilon\|\leq1$,
\begin{align}
\label{f77}
&\alpha_3(|x|)-\alpha_4(|w|) +
\frac{\partial{V}}{\partial{x}}(x)
\Big(\Zbb ^\top +\mathbf{\bar{Q}}^\frac{1}{2}\Upsilon^\top \mathbf{\bar{A}}^{-\frac{1}{2}} \Big) 
\smat{
Z(x)\\W(x)(k(x)+w)}
\leq 0,
\end{align}
then, by \eqref{set_overapproximation}, \eqref{problem_actuator_b} holds for all $x$, $w$ and all $\AB \in\bar{\mathcal{I}} \supseteq \mathcal{I}$.
Since \eqref{problem_actuator_a} and \eqref{problem_actuator_b} correspond to \eqref{ISS_lyap_fun_prop:bounds} and \eqref{ISS_lyap_fun_prop:dissip_ineq}, the continuously differentiable $V$ is, by Definition~\ref{def:ISS_lyap_fun}, an ISS-Lyapunov function for \eqref{theorem_disturbance_sys} for all $\AB \in \overline{\mathcal{I}}$; since $\AB[\star] \in \mathcal{I} \subseteq \overline{\mathcal{I}}$ by Assumption~\ref{assumpt:noise}, this and Fact~\ref{fact:equiv_iss} would prove the statement. 
We show then that by \eqref{biconvex_ISS_w:dissip}, \eqref{f77} holds for all $x$, $w$ and all $\Upsilon$ with $\|\Upsilon\| \leq 1$.
We rewrite \eqref{f77} as
\begin{align}
& \alpha_3(|x|)-\alpha_4(|w|) 
+ \frac{\partial{V}}{\partial{x}}(x)
\Zbb ^\top 
\smat{
Z(x)\\W(x)(k(x)+w)} 
\label{f7} \\
& 
+ \Tp\bigg\{ \frac{1}{2} \smat{
Z(x)\\W(x)(k(x)+w)}^\top
\mathbf{\bar{A}}^{-\frac{1}{2}}
\Upsilon \mathbf{\bar{Q}}^\frac{1}{2}
\frac{\partial{V}}{\partial{x}}(x)^\top  \bigg\}
\leq 0 . \notag
\end{align}
Since $\mu>0$ and $s_\lambda$ is an SOS polynomial, \eqref{biconvex_ISS_w:lambda} implies $\lambda(x,w)>0$ for all $x$ and $w$. So, by completing the square \cite[Lemma~2]{chen2023data}, \eqref{f7} is valid for all $x$, $w$ and all $\Upsilon$ with $\|\Upsilon\| \le 1$  if, for all $x$, $w$,
\begin{align}
&\alpha_3(|x|)-\alpha_4(|w|)+\frac{\partial{V}}{\partial{x}}(x)\boldsymbol{\Zbb}^\top \smat{Z(x)\\W(x)(k(x)+w)} +\frac{1}{2\lambda(x,w)}\frac{\partial{V}}{\partial{x}}(x)\mathbf{\bar{Q}}^\frac{1}{2}\mathbf{\bar{Q}}^\frac{1}{2}\frac{\partial{V}}{\partial{x}}(x)^\top \label{f8}
\\
&+\frac{\lambda(x,w)}{2}\left[\begin{smallmatrix}Z(x)\\W(x)(k(x)+w)\end{smallmatrix}\right]^\top \mathbf{\bar{A}}^{-\frac{1}{2}}\mathbf{\bar{A}}^{-\frac{1}{2}}\smat{Z(x)\\W(x)(k(x)+w)} \leq 0. \notag
\end{align}
Applying Schur complement \cite[p.~28]{boyd1994linear} to \eqref{f8} yields that, for all $x$, $w$,
\begin{equation}\label{Th2-1}
\bmat{
\left\{
\begin{matrix}
\alpha_3(| x |)-\alpha_4(|w|)\\+\frac{\partial{V}}{\partial{x}}(x)\Zbb^\top 
\smat{Z(x)\\W(x)(k(x)+w)}
\end{matrix}
\right\} &\star&\star\\
\mathbf{\bar{Q}}^\frac{1}{2}\frac{\partial{V}}{\partial{x}}(x)^\top & -2\lambda(x,w) I & \star\\ 
\lambda(x,w) \mathbf{\bar{A}}^{-\frac{1}{2}}
\smat{Z(x)\\W(x)(k(x)+w)} & 0 & -2\lambda(x,w) I 
} \preceq 0,
\end{equation}
which is implied by \eqref{biconvex_ISS_w:dissip}.
\end{proof}

Some comments on Theorem~\ref{thm:biconvex_ISS_w} are appropriate. 
\begin{enumerate}[left=-3pt,label=\alph*)]
\item \eqref{biconvex_ISS_w:alpha} ensures that $\alpha_1, \dots, \alpha_4$ are class $\mathcal{K}_{\infty}$ functions by Lemma~\ref{lem:classk}.
We consider even powers of $r$ in these comparison functions in~\eqref{biconvex_ISS_w:alpha} because such powers cancel the square roots in $\sqrt{\sum_{i=1}^n x_i^2}=|x|$ or $\sqrt{\sum_{i=1}^n w_i^2}=|w|$ so that $\alpha_1(|x|)$, $\alpha_2(|x|)$, $\alpha_3(|x|)$, $\alpha_4(|w|)$ in \eqref{biconvex_ISS_w:bounds_V} and \eqref{biconvex_ISS_w:lambda} are polynomials in the components of $x$ and $w$, and, thus, SOS tools can be applied to solve \eqref{biconvex_ISS_w}. 
By allowing some coefficients of functions $\alpha_i$, $i=1, \dots, 4$, to be zero, we let the SOS program \eqref{biconvex_ISS_w} design the maximum degree of the $\alpha_i$'s, which is only required to be nongreater than a positive integer $N_i$ of our choice.
\item \eqref{biconvex_ISS_w:bounds_V} entails lower\slash upper bounds on the ISS-Lyapunov function $V$, as in \eqref{problem_actuator_a}.
\item \eqref{biconvex_ISS_w:lambda} entails positivity of the multiplier $\lambda$ as needed to complete a square in the proof.
\item \eqref{biconvex_ISS_w:dissip} corresponds to guaranteeing the dissipativity-like inequality \eqref{problem_actuator_b} for all polynomial dynamics that are consistent with data and uses the quantities $\Zbb$, $\bar{\mathbf{Q}}$, $\bar{\mathbf{A}}$ that were obtained from $\{u(t_i), x(t_i), \dot{x}(t_i)\}_{i=0}^{T-1}$ by solving \eqref{overapp}.
\end{enumerate}

Theorem~\ref{thm:biconvex_ISS_w} effectively solves Problem~\ref{problem:actuator_disturbance} since $\mathcal{I} \subseteq \bar{\mathcal{I}}$ by Fact~\ref{fact:overapp_set}.

\begin{remark}
Due to \eqref{biconvex_ISS_w:alpha}, the bounds $\alpha_1$ and $\alpha_2$ take a specific form, which may be unnecessarily specific.
Instead, for a positive definite and radially unbounded $\ell$ and an SOS polynomial $s_\ell$, one can ask that, for all $x$,
\begin{align}
\label{biconvex_ISS_w:bounds_V_alternative}
V(x)-\ell(x)=s_\ell(x).
\end{align}
Then, \cite[Lemma 4.3]{khalil2002nonlinear} guarantees the existence of class $\mathcal{K}_\infty$ functions $\underline{\alpha}$ and $\overline{\alpha}$ such that for all $x$, $\underline{\alpha}(|x|) \le V(x) \le \overline{\alpha}(|x|)$,
so \eqref{biconvex_ISS_w:bounds_V_alternative} effectively replaces \eqref{biconvex_ISS_w:bounds_V}.
On the other hand, knowing the expressions of $\underline{\alpha}$ or $\overline{\alpha}$ may be relevant, e.g., to obtain the actual bound on solutions in~\eqref{ISS_bound_solutions} from the dissipativity-like inequality in~\eqref{ISS_lyap_fun_prop:dissip_ineq} on the ISS-Lyapunov function.
The aforementioned approach does not lead to an explicit expression for $\underline{\alpha}$ or $\overline{\alpha}$ whereas Theorem~\ref{thm:biconvex_ISS_w} does.
\end{remark}

\subsection{ISS with respect to process disturbances}\label{subsec:process_noise}

In this subsection, we consider the system
\begin{equation}\label{system_op_nl_pro_noise}
    \dot{x}=A_\star Z(x)+B_\star W(x)u+d =: f^{\tu{p}}_{A_\star,B_\star}(x,d,u)
\end{equation}
where $d\in\mathbb{R}^n$ is the process disturbance.
In the next theorem, we provide conditions to find a controller $u=k(x)$ that makes the closed loop input-to-stable with respect to $d$ for all possible dynamics consistent with data, as certified by an ISS-Lyapunov function and comparison functions.

\begin{theorem}
\label{thm:biconvex_ISS_d}
For data points $\{u(t_i), x(t_i), \dot{x}(t_i)\}_{i=0}^{T-1}$ and under Assumptions~\ref{assumpt:sys}-\ref{assumpt:noise}, let the optimization program in \eqref{overapp} be feasible. Suppose there exist a scalar $\mu>0$, a polynomial vector $k$ with $k(0)=0$, polynomials $\alpha_1$, \dots, $\alpha_4$, $V$, $\lambda$, SOS polynomials $s_{\alpha_1}$, $s_{\alpha_2}$, $s_\lambda$, and an SOS polynomial matrix $S_\nabla$ such that for all $r$, $x$ and $d$%
\begin{subequations}
\label{biconvex_ISS_d}
\begin{align}
&\alpha_i(r)=\sum_{k=1}^{N_i} c_{i k} r^{2 k}, c_{i 1} \ge 0, \dots, c_{i N_i} \ge 0, \sum_{k=1}^{N_i} c_{i k} \ge \mu  \qquad \text {for } i=1, \dots, 4
\label{biconvex_ISS_d:alpha}\\
&V(x)-\alpha_1(|x|) = s_{\alpha_1}(x),\quad \alpha_2(|x|)-V(x) =s_{\alpha_2}(x),\label{biconvex_ISS_d:bounds_V}\\
&\lambda(x,d)-\mu=s_\lambda(x,d),\label{biconvex_ISS_d:lambda}\\
& 
\bmat{
\left\{\begin{matrix}
\alpha_3(|x|)-\alpha_4(|d|)\\
\!+\frac{\partial{V}}{\partial{x}}(x)\Big(\boldsymbol{\Zbb}^\top \smat{Z(x)\\W(x)k(x)} +d\Big) \!\end{matrix}\right\}&\star&\star\\\mathbf{\bar{Q}}^\frac{1}{2}\frac{\partial{V}}{\partial{x}}(x)^\top &-2\lambda(x,d) I&\star\\ \lambda(x,d) \mathbf{\bar{A}}^{-\frac{1}{2}}
\smat{Z(x)\\W(x)k(x)} &0&-2\lambda(x,d) I
}\!
=-S_\nabla(x,d).\label{biconvex_ISS_d:dissip}
\end{align}
\end{subequations}
Then, the system
\begin{align*}
\dot{x}=
\bmat{A & B}
\bmat{Z(x)\\ W(x) k(x)} +d = f^{\tu{p}}_{A,B}(x,d,k(x))
\end{align*}
is ISS with respect to process disturbance $d$ for all $\AB \in \bar{\mathcal{I}}$, and in particular for $\AB[\star]$.
\end{theorem}
\begin{proof}
Following the same arguments in the proof of Theorem~\ref{thm:biconvex_ISS_w} and considering the process disturbance as the exogenous input, it is easy to verify that \eqref{biconvex_ISS_d:alpha} and \eqref{biconvex_ISS_d:bounds_V} imply \eqref{problem_process_a}. Resorting to Lemma~\ref{lem:classk} and Schur complement, \eqref{biconvex_ISS_d:lambda} and \eqref{biconvex_ISS_d:dissip} imply \eqref{problem_process_b} for all $\AB \in \bar{\mathcal{I}} \supseteq \mathcal{I}$.
Hence, the conclusion holds by Definition~\ref{def:ISS_lyap_fun}.
\end{proof}
Theorem~\ref{thm:biconvex_ISS_d} effectively solves Problem~\ref{problem:process_noise} since $\mathcal{I} \subseteq \bar{\mathcal{I}}$ by Fact~\ref{fact:overapp_set}.

\begin{remark}
If $\mu$, $k$, $\alpha_1, \dots, \alpha_4$, $V$, $\lambda$, $s_{\alpha_1}$, $s_{\alpha_2}$, $s_\lambda$, $S_\nabla$ are as hypothesized in Theorem~\ref{thm:biconvex_ISS_w} or in Theorem~\ref{thm:biconvex_ISS_d}, then $\mu$, $k$, $\alpha_1, \dots, \alpha_3$, $V$, $\lambda(\cdot,0)$, $s_{\alpha_1}$, $s_{\alpha_2}$, $s_\lambda(\cdot,0)$, $S_\nabla(\cdot,0)$ satisfy the hypothesis of Theorem~\ref{thm:bilinear_GAS}.
This implication is a ``data-driven counterpart'' of the well-known fact that input-to-state stability with respect to an exogeneous input implies global asymptotic stability for the exogenous input set to zero (the so-called $0$-GAS).
In this sense, the feasibility set of the SOS program associated with Theorem~\ref{thm:biconvex_ISS_w} or Theorem~\ref{thm:biconvex_ISS_d} is a subset of the feasibility set of the SOS program associated with Theorem~\ref{thm:bilinear_GAS}.
\end{remark}

Admittedly, the SOS programs \eqref{biconvex_ISS_w} and \eqref{biconvex_ISS_d} are biconvex \cite[Def.~1.3]{gorski2007biconvex}, and thus nonconvex, 
due to the products of decision variables $V$, $k$ and $\lambda$, $k$ in \eqref{biconvex_ISS_w:dissip} and \eqref{biconvex_ISS_d:dissip}.
To address this nonconvexity, we adopt the widespread alternate approach \cite{jarvis2005control}, \cite{majumdar2013control} where, in a first step, we fix $k$ and solve \eqref{biconvex_ISS_w} and \eqref{biconvex_ISS_d}, which have so become convex with respect to the remaining decision variables and, in a second step, we fix $V$, $\lambda$ and solve \eqref{biconvex_ISS_w} and \eqref{biconvex_ISS_d}, which have so become convex with respect to the remaining decision variables.
It is also important to recognize that the effectiveness of this approach depends on finding an initial guess for $k$ or $V$, $\lambda$.
Overall, this alternate approach has proven to be successful in our numerical examples of Section~\ref{sec:exper}.

\subsection{Construction of convex SOS programs for ISS}\label{subsec:convex}

In this section, we pursue a different approach that circumvents the biconvexity and the need of an initial guess of the approach proposed in Theorems~\ref{thm:biconvex_ISS_w} and \ref{thm:biconvex_ISS_d}.
This different approach allows constructing convex SOS programs to synthesize an ISS-Lyapunov function and a controller, as required by Problems~\ref{problem:actuator_disturbance} and \ref{problem:process_noise}.
This is achieved by employing polynomial parametrizations of the ISS-Lyapunov function and controller with the forms $Z(x)^\top P^{-1}Z(x)$ and $Y(x)P^{-1}Z(x)$, for some matrix $P = P^\top \succ0$ and polynomial matrix $Y$.
This parametrization is more specific than those in Theorems~\ref{thm:biconvex_ISS_w} and \ref{thm:biconvex_ISS_d}, where the ISS-Lyapunov function and controller are searched among polynomials without the aforementioned specific structure.
This parametrization is borrowed from~\cite{prajna2004nonlinear,guo2021data}, where global asymptotic stabilization was considered.

Moreover, while the function library $Z$ may have relatively many monomials to match the actual vector field $f_\star$, see Remark~\ref{rmk:regressor}, it may be unnecessary and potentially counterproductive to have all these monomials also in the ISS-Lyapunov function and controller parametrized as $Z(x)^\top P^{-1}Z(x)$ and $Y(x)P^{-1}Z(x)$, with $P$ and $Y$ as above.
Instead, one can use a polynomial vector $\hat{Z}$ of smaller size than $Z$ and parametrize ISS-Lyapunov function and controller as $\hat{Z}(x)^\top P^{-1} \hat{Z}(x)$ and $Y(x)P^{-1} \hat{Z}(x)$.
This may also have the benefit of reducing the order of the SOS programs to be solved.
The function library $\hat{Z}$ needs to satisfy the next assumption.

\begin{assumption}\label{assumpt:Zhat}
For function library $Z$, there exist a polynomial vector $\hat{Z}\colon \real^n \to \real^{\hat{N}}$ and a polynomial matrix $H \colon \real^n \to \real^{N\times {\hat{N}}}$ such that $\hat{Z}(x)=0$ if and only if $x=0$, $Z(x)=H(x)\hat{Z}(x)$ for all $x$, and $x\mapsto |\hat{Z}(x)|$ is radially unbounded, i.e., $|\hat{Z}(x)|\rightarrow\infty$ as $|x|\rightarrow\infty$.
\end{assumption}

In Assumption~\ref{assumpt:Zhat}, we recover $Z = \hat{Z}$ if we set $x \mapsto H(x)=I_N$.
Also, the conditions of Assumption~\ref{assumpt:Zhat} on the choice of $\hat{Z}$ and $H$ are easy to satisfy. 
For instance, if we take the vector $x=(x_1,\dots,x_n)$ as a part of the vector $\hat{Z}(x)$, there always exists some $H$ such that Assumption~\ref{assumpt:Zhat} holds. 
The choice of $H$ is not unique. 
For example, for the function library $Z(x)=(x_1^2,x_1x_2,x_2^2)$, if we take $\hat{Z}(x)=(x_1,x_2)$, then $H(x)=\smat{x_1&0\\x_2&0\\0&x_2}$ or $H(x)=\smat{x_1&0\\0&x_1\\0&x_2}$ can be chosen so that $Z(x)=H(x)\hat{Z}(x)$ for all $x$.

By substituting $Z$ with $H\hat{Z}$ thanks to Assumption~\ref{assumpt:Zhat}, the input-affine polynomial systems \eqref{system_op_nl_act_dist} and \eqref{system_op_nl_pro_noise} are rewritten as
\begin{equation}\label{convex_sys_d}
\dot{x}= f^{\tu{a}}_{A_\star,B_\star}(x,w,u) = A_\star H(x)\hat{Z}(x)+B_\star W(x)(u+w)
\end{equation}
and
\begin{equation}\label{convex_sys_w}
\dot{x}= f^{\tu{p}}_{A_\star,B_\star}(x,d,u) = A_\star H(x)\hat{Z}(x)+B_\star W(x)u+d.
\end{equation}

For~\eqref{convex_sys_d}, we present a theorem for data-driven convex design of a controller and an ISS-Lyapunov function that ensure the closed-loop system is input-to-state stable with respect to actuator disturbances.

\begin{theorem}
\label{thm:convex_ISS_w}
For data points $\{u(t_i), x(t_i), \dot{x}(t_i)\}_{i=0}^{T-1}$ and under Assumptions~\ref{assumpt:sys}, \ref{assumpt:noise} and \ref{assumpt:Zhat}, let the optimization program in \eqref{overapp} be feasible. 
For given symmetric polynomial matrix $\Xi$ and scalar $\varepsilon >0$,
suppose there exist a matrix $P=P^\top \succ 0$, polynomial matrices $Y$ and $\Gamma$, symmetric polynomial matrix $\Theta$, a scalar $\eta>0$, a polynomial $\lambda$, an SOS polynomial $s_\lambda$, SOS polynomial matrices $S_\Theta$ and $S_\partial$ such that for all $r$, $x$ and $w$,
\begin{subequations}%
\label{convex_ISS_w}
\begin{align}
&\Gamma(r)=\sum_{k=0}^{N} C_{k} r^{2k}, C_0 = C_0^\top \succeq 0, \dots, C_N = C_N^\top \succeq 0, \sum_{k=0}^{N} C_{k} \succeq \varepsilon I,
\label{convex_ISS_w:gamma}\\
&\lambda(x,w)-\varepsilon=s_\lambda(x,w),\label{convex_ISS_w:lambda}\\
&\Theta(x)-\eta\Xi(x)=S_\Theta(x),\label{convex_ISS_w:Theta}\\
&
\bmat{
\left\{
\begin{matrix}
\Tp\bigg(\smat{H(x)P\\W(x)Y(x)}^\top \Zbb\frac{\partial \hat{Z}(x)}{\partial x}^\top \bigg) \\
+\Theta(x) +\lambda(x,w)\frac{\partial \hat{Z}(x)}{\partial x}\bar{\mathbf{Q}}\frac{\partial \hat{Z}(x)}{\partial x}^\top 
\end{matrix}
\right\}\hspace*{-2mm} &\star&\star\\[6.5mm]
\smat{0\\W(x)}^\top \Zbb \frac{\partial \hat{Z}(x)}{\partial x}^\top &-\Gamma(|w|) &\star\\        
\smat{H(x)P\\W(x)Y(x)} & \smat{0\\W(x)} & -\lambda(x,w)\bar{\mathbf{A}}
} = -S_\partial(x,w).\label{convex_ISS_w:dissipation}
\end{align}
\end{subequations}
Then, if the function 
\begin{align}
\label{convex_ISS_w:b_func}
b(\cdot):=\hat{Z}(\cdot)^{\top}P^{-1}\Xi(\cdot)P^{-1}\hat{Z}(\cdot)
\end{align}
is positive definite and radially unbounded, 
the closed-loop system
\begin{align*}
\dot{x}= \AB \bmat{H(x)\hat{Z}(x)\\ W(x) \big( Y(x)P^{-1}\hat{Z}(x) +w \big)}  = f^{\tu{a}}_{A,B}(x,w,Y(x)P^{-1}\hat{Z}(x))
\end{align*}
is ISS with respect to the actuator disturbance $w$ for all $\AB \in \bar{\mathcal{I}}$, and in particular for $\AB[\star]$.
\end{theorem}
\begin{proof}
\begingroup%
\thinmuskip=1mu plus 1mu minus 1mu
\medmuskip=2mu plus 2mu minus 2mu
\thickmuskip=3mu plus 3mu minus 3mu
Let us consider $x \mapsto V(x):=\hat{Z}(x)^\top P^{-1}\hat{Z}(x)$ as an ISS-Lyapunov function.
By Assumption~\ref{assumpt:Zhat}, $V \colon \real^n \to \real$ is continuous since $\hat{Z}$ is a polynomial vector; $V$ is positive definite because $\hat{Z}(x) = 0$ if and only if $x=0$ and $P = P^\top \succ 0$; $V$ is radially unbounded because $\hat{Z}$ is radially unbounded.
Hence, by~\cite[Lemma 4.3]{khalil2002nonlinear}, there exist class $\mathcal{K}_\infty$ functions $\alpha_1$ and $\alpha_2$ such that for all $x$
\begin{equation}
\label{convex_ISS_w:sandwich_V}
   \alpha_1(|x|) \le V(x) \le \alpha_2(|x|). 
\end{equation}
For all $x$, $BW(x)=\AB \smat{0\\W(x)}$ so that we have
\begin{align*}
\dot{x} & = f^{\tu{a}}_{A,B}(x,w,Y(x)P^{-1}\hat{Z}(x)) \\
& = \AB \smat{H(x)P\\ W(x)Y(x)} P^{-1}\hat{Z}(x)+ \AB \smat{0\\W(x)}w =: F^{\tu{a}}_{A,B}(x,w).
\end{align*}
The derivative of $V$ along solutions of the closed-loop system is
\begin{align*}
& \langle \nabla V(x), F^{\tu{a}}_{A,B}(x,w)\rangle = 2\hat{Z}(x)^\top P^{-1}\tfrac{\partial \hat{Z}(x)}{\partial x}\AB
\Big(
\smat{
H(x)P\\
W(x)Y(x)
}
P^{-1}\hat{Z}(x)
+
\smat{0\\W(x)} w
\Big)\\
& = 
\hat{Z}(x)^\top \! P^{-1}\!
\Tp\Big(\!
\tfrac{\partial \hat{Z}(x)}{\partial x}\AB \!
\smat{
H(x)P\\
W(x)Y(x)}
\!\Big)
P^{-1}\hat{Z}(x)
 +
2 \hat{Z}(x)^\top\! P^{-1} \tfrac{\partial \hat{Z}(x)}{\partial x}\AB  \!\smat{0\\W(x)} w\\
& = 
\smat{P^{-1}\hat{Z}(x)\\ w}^\top
\smat{
\Tp\Big(
\smat{
H(x)P\\
W(x)Y(x)}^\top \AB^\top \frac{\partial \hat{Z}(x)}{\partial x}^\top
\Big) & \frac{\partial \hat{Z}(x)}{\partial x}\AB \smat{0\\W(x)}  \\
\smat{0\\W(x)}^\top \AB^\top \frac{\partial \hat{Z}(x)}{\partial x}^\top & 0
}
\smat{P^{-1}\hat{Z}(x)\\ w}.
\end{align*}
Define 
\begin{align}
\label{convex_ISS_w:a}
a(x):=\hat{Z}(x)^\top P^{-1}\Theta(x)P^{-1}\hat{Z}(x).
\end{align}
We would like to show that
\begin{align}
& \forall x,w, \AB \in \bar{\mathcal{I}}, \label{convex_ISS_w:diss_ineq_almost_proof}\\
& 0 \ge \langle \nabla V(x), F^{\tu{a}}_{A,B}(x,w)\rangle + a(x) -  w^\top \Gamma(|w|) w \notag \\
& = 
\smat{P^{-1}\hat{Z}(x)\\ w}^\top
\smat{
\Tp\Big(
\smat{
H(x)P\\
W(x)Y(x)}^\top \AB^\top \frac{\partial \hat{Z}(x)}{\partial x}^\top
\Big) + \Theta(x) & \star  \\
\smat{0\\W(x)}^\top \AB^\top \frac{\partial \hat{Z}(x)}{\partial x}^\top & - \Gamma(|w|)
}
\smat{P^{-1}\hat{Z}(x)\\ w}. \notag
\end{align}
This holds if
\begin{align*}
& \forall x,w, \AB \in \bar{\mathcal{I}}, \quad 0 \succeq 
\smat{
\Tp\Big(
\smat{
H(x)P\\
W(x)Y(x)}^\top \AB^\top \frac{\partial \hat{Z}(x)}{\partial x}^\top
\Big) + \Theta(x) & \frac{\partial \hat{Z}(x)}{\partial x}\AB \smat{0\\W(x)}  \\
\smat{0\\W(x)}^\top \AB^\top \frac{\partial \hat{Z}(x)}{\partial x}^\top & - \Gamma(|w|)
}.
\end{align*}
By the equivalent parametrization of the set $\bar{\mathcal{I}}$ in~\eqref{set_overapproximation}, this holds if
\begin{align*}
& \forall x,w, \Upsilon  \text{ with } \| \Upsilon \| \le 1, \\
& 0 \succeq 
\smat{
\Tp\Big(
\smat{
H(x)P\\
W(x)Y(x)}^\top \Zbb \frac{\partial \hat{Z}(x)}{\partial x}^\top
\Big) + \Theta(x) & \star  \\
\smat{0\\W(x)}^\top \Zbb \frac{\partial \hat{Z}(x)}{\partial x}^\top & -\Gamma(|w|)
}+
\smat{
\Tp\Big(
\smat{
H(x)P\\
W(x)Y(x)}^\top \bar{\mathbf{A}}^{-\h} \Upsilon \bar{\mathbf{Q}}^\h \frac{\partial \hat{Z}(x)}{\partial x}^\top
\Big)  & \star  \\
\smat{0\\W(x)}^\top \bar{\mathbf{A}}^{-\h} \Upsilon \bar{\mathbf{Q}}^\h \frac{\partial \hat{Z}(x)}{\partial x}^\top & 0
}\\
& \!\! = \!\! \smat{
\Tp\Big(
\smat{
H(x)P\\
W(x)Y(x)}^\top \Zbb \frac{\partial \hat{Z}(x)}{\partial x}^\top \!
\Big) + \Theta(x) & \star  \\
\smat{0\\W(x)}^\top \Zbb \frac{\partial \hat{Z}(x)}{\partial x}^\top & - \Gamma(|w|)
}\!\!\!+\!\!
\Tp
\left(\!\!
\smat{
\smat{H(x)P\\W(x)Y(x)}^\top \bar{\mathbf{A}}^{-\h}\\
\smat{0\\W(x)}^\top \bar{\mathbf{A}}^{-\h}
}\!
\Upsilon
\smat{\bar{\mathbf{Q}}^\h \frac{\partial \hat{Z}(x)}{\partial x}^\top & 0}\!\!\!
\right)\!.
\end{align*}
Since $\varepsilon>0$ and $s_\lambda$ is an SOS polynomial, \eqref{convex_ISS_w:lambda} implies that $\lambda(x,w)>0$ for all $x$ and $w$. 
Then, the previous condition holds if
\begin{align*}
& \forall x,w, \Upsilon  \text{ with } \| \Upsilon \| \le 1, \\
& \smat{
\Tp\Big(
\smat{
H(x)P\\
W(x)Y(x)}^\top \Zbb \frac{\partial \hat{Z}(x)}{\partial x}^\top
\Big) + \Theta(x) & \star  \\
\smat{0\\W(x)}^\top \Zbb \frac{\partial \hat{Z}(x)}{\partial x}^\top & -\Gamma(|w|)
}+
\Tp
\left(
\smat{
\smat{H(x)P\\W(x)Y(x)}^\top \bar{\mathbf{A}}^{-\h}\\
\smat{0\\W(x)}^\top \bar{\mathbf{A}}^{-\h}
}
\Upsilon
\smat{\bar{\mathbf{Q}}^\h \frac{\partial \hat{Z}(x)}{\partial x}^\top & 0}
\right) \\
& \preceq 
\smat{
\Tp\Big(
\smat{
H(x)P\\
W(x)Y(x)}^\top \Zbb \frac{\partial \hat{Z}(x)}{\partial x}^\top
\Big) + \Theta(x) & \star  \\
\smat{0\\W(x)}^\top \Zbb \frac{\partial \hat{Z}(x)}{\partial x}^\top & -\Gamma(|w|)
} \\
& +\frac{1}{\lambda(x,w)}
\smat{
\smat{H(x)P\\W(x)Y(x)}^\top \\
\smat{0\\W(x)}^\top
}
\bar{\mathbf{A}}^{-1}
\smat{
\smat{H(x)P\\W(x)Y(x)}^\top \\
\smat{0\\W(x)}^\top
}^\top
+ \lambda(x,w)
\smat{
\frac{\partial \hat{Z}(x)}{\partial x}\\
0}
\bar{\mathbf{Q}}
\smat{
\frac{\partial \hat{Z}(x)}{\partial x}\\
0}^\top \preceq 0
\end{align*}
where the inequality follows by completing the square \cite[Lemma~2]{chen2023data} and $\| \Upsilon \| \le 1$.
By Schur complement \cite[p.~28]{boyd1994linear}, the previous condition holds if
\begin{align*}
& \forall x,w, \quad   
\smat{
\left\{
\begin{smallmatrix}
\Tp\Big(
\smat{
H(x)P\\
W(x)Y(x)}^\top \Zbb \frac{\partial \hat{Z}(x)}{\partial x}^\top
\Big) \\
+ \Theta(x) + \lambda(x,w) \frac{\partial \hat{Z}(x)}{\partial x} \bar{\mathbf{Q}} \frac{\partial \hat{Z}(x)}{\partial x}^\top
\end{smallmatrix}
\right\} & \frac{\partial \hat{Z}(x)}{\partial x}\Zbb^\top \smat{0\\W(x)} & \smat{H(x)P\\W(x)Y(x)}^\top \\
\smat{0\\W(x)}^\top \Zbb \frac{\partial \hat{Z}(x)}{\partial x}^\top & -\Gamma(|w|) & \smat{0\\W(x)}^\top\\
\smat{H(x)P\\W(x)Y(x)} & \smat{0\\W(x)} & - \lambda(x,w) \bar{\mathbf{A}}
}\preceq 0,
\end{align*}
and this is implied by~\eqref{convex_ISS_w:dissipation}.
We have then shown that \eqref{convex_ISS_w:diss_ineq_almost_proof} holds.

Consider the function $a$ defined in~\eqref{convex_ISS_w:a}.
Since $\hat{Z}(0) = 0$ by Assumption~\ref{assumpt:Zhat}, $a(0) = 0$ and, by~\eqref{convex_ISS_w:Theta}, 
\begin{align*}
& \forall x, \quad  a(x) = \hat{Z}(x)^\top P^{-1}(\eta\Xi(x) +S_\Theta(x)) P^{-1}\hat{Z}(x) \ge \eta \hat{Z}(x)^\top P^{-1}\Xi(x) P^{-1}\hat{Z}(x)=\eta b(x)
\end{align*}
where $\eta> 0$ and $b$ in~\eqref{convex_ISS_w:b_func} is positive definite and radially unbounded by hypothesis.
Hence, $a \colon \real^n \to \real$ is continuous, positive definite and radially unbounded.
By these properties and \cite[Lemma 4.3]{khalil2002nonlinear}, there exists a class $\mathcal{K}_\infty$ function $\alpha_3$ such that $\alpha_3(|x|) \le a(x)$ for all $x$.
Moreover, by~\eqref{convex_ISS_w:gamma},
\begin{align*}
\forall w,\quad w^\top \Gamma(|w|) w 
& = \sum_{k=0}^N w^\top C_k w |w|^{2k}  \\ 
& \le \sum_{k=0}^N \lambda_{\max}(C_k) |w|^{2(k+1)}
= \sum_{k=1}^{N+1} \lambda_{\max}(C_{k-1}) |w|^{2k} =: \alpha_4(|w|)
\end{align*}
where $\lambda_{\max}(C_0) \ge 0$, \dots, $\lambda_{\max}(C_N) \ge 0$ and $\sum_{k=1}^{N+1} \lambda_{\max}(C_{k-1}) > 0$ by~\eqref{convex_ISS_w:gamma}.
Hence, $\alpha_4$ belongs to class $\mathcal{K}_\infty$ by Lemma~\ref{lem:classk}.
We then have, from~\eqref{convex_ISS_w:diss_ineq_almost_proof}, that
\begin{align}\label{convex_ISS_w:diss_ineq}
& \forall x,w, \AB \in \bar{\mathcal{I}},  \quad \langle \nabla V(x), F^{\tu{a}}_{A,B}(x,w)\rangle \le -a(x) +  w^\top \Gamma(|w|) w \le -\alpha_3(|x|) + \alpha_4(|w|). \notag
\end{align}
Therefore, we conclude by~\eqref{convex_ISS_w:sandwich_V}, \eqref{convex_ISS_w:diss_ineq} and Definition~\ref{def:ISS_lyap_fun} that for all $\AB \in \bar{\mathcal{I}}$, $V$ is an ISS-Lyapunov function for $\dot{x} = F^{\tu{a}}_{A,B}(x,w)$ and by Fact~\ref{fact:equiv_iss} that for all $\AB \in \bar{\mathcal{I}}$, and in particular for $\AB[\star] \in \bar{\mathcal{I}}$, the system $\dot{x} = F^{\tu{a}}_{A,B}(x,w)$ is input-to-state stable with respect to $w$.
\endgroup
\end{proof}

Theorem~\ref{thm:convex_ISS_w} solves Problem~\ref{problem:actuator_disturbance}, whereby $\mathcal{I} \subseteq \bar{\mathcal{I}}$ by Fact~\ref{fact:overapp_set} and the comparison functions of Problem~\ref{problem:actuator_disturbance} can be obtained as shown in the proof of Theorem~\ref{thm:convex_ISS_w}.
Notably, \eqref{convex_ISS_w} does not include any product of decision variables and is thus convex, unlike the biconvex \eqref{biconvex_ISS_w}.
To apply Theorem~\ref{thm:convex_ISS_w}, one needs to choose the design parameter $\Xi$, solve \eqref{convex_ISS_w} and, for the obtained $P$, verify that $b$ in~\eqref{convex_ISS_w:b_func} is positive definite and radially unbounded.
In~\eqref{convex_ISS_w:Theta}, $\eta > 0$ is a decision variable and this allows us to choose the design parameter $\Xi$ modulo a scaling, which provides some freedom in the choice of $\Xi$.

In terms of guidelines, it is sensible to select $\Xi$ so that $\Xi(x) \succeq 0$ for all $x$, since an ultimate goal is that $b$ in~\eqref{convex_ISS_w:b_func} is positive definite.
As a possible selection, $\Xi$ equal to $\hat{Z} \hat{Z}^\top$ ensures that for any $P = P^\top \succ 0$ returned by~\eqref{convex_ISS_w}, $b(\cdot) =  \hat{Z}(\cdot)^{\top}P^{-1}\hat{Z}(\cdot)\hat{Z}(\cdot)^\top P^{-1}\hat{Z}(\cdot) = |P^{-1/2} \hat{Z}(\cdot)|^4$ is positive definite and radially unbounded by Assumption~\ref{assumpt:Zhat}.

For the cases when the selection of $\Xi$ is not straightforward, one can consider an additional polynomial matrix $\hat{\Gamma}$ such that for all $r$,
\begin{align}
\label{convex_ISS_w:hatGamma}
\hat{\Gamma}(r)=\sum_{k=0}^{\hat{N}} \hat{C}_{k} r^{2k}, \hat{C}_0 = \hat{C}_0^\top \succeq 0, \dots, \hat{C}_{\hat{N}} = \hat{C}_{\hat{N}}^\top \succeq 0, \sum_{k=0}^{\hat{N}} \hat{C}_{k} \succeq \varepsilon I
\end{align}
and take $\Theta(x) = \hat{\Gamma}(|x|)$ to satisfy \eqref{convex_ISS_w:gamma}, \eqref{convex_ISS_w:lambda} and \eqref{convex_ISS_w:dissipation}.
If \eqref{convex_ISS_w:hatGamma}, \eqref{convex_ISS_w:gamma}, \eqref{convex_ISS_w:lambda} and the so-modified \eqref{convex_ISS_w:dissipation} hold for all $r$, $x$ and $w$, then (i)~\eqref{convex_ISS_w:Theta} holds with $\Xi(\cdot) = \hat{\Gamma}(|\cdot|)$, $\eta =1$, $S_\Theta(\cdot) = 0$; (ii)~$b$ in~\eqref{convex_ISS_w:b_func} is positive definite and radially unbounded.
This selection is effectively a special case of Theorem~\ref{thm:convex_ISS_w} that does not require tuning $\Xi$.

After actuator disturbances, we consider process disturbances.
For~\eqref{convex_sys_w}, we present a theorem for data-driven convex design of a controller and an ISS-Lyapunov function that ensures the closed-loop system is input-to-state stable with respect to process disturbances.

\begin{theorem}
\label{thm:convex_ISS_d}
For data points $\{u(t_i), x(t_i), \dot{x}(t_i)\}_{i=0}^{T-1}$ and under Assumptions~\ref{assumpt:sys}, \ref{assumpt:noise} and \ref{assumpt:Zhat}, let the optimization program in \eqref{overapp} be feasible. 
For given symmetric polynomial matrix $\Xi$ and scalar $\varepsilon>0$, suppose there exist a matrix $P=P^\top \succ 0$, polynomial matrices $Y$ and $\Gamma$, a symmetric polynomial matrix $\Theta$, a scalar $\eta>0$, a polynomial $\lambda$, an SOS polynomial $s_\lambda$, SOS polynomial matrices $S_\Theta$ and $S_\partial$ such that for all $r$, $x$ and $d$%
\begin{subequations}
\label{convex_ISS_d}
\begin{align}
&\Gamma(r)=\sum_{k=0}^{N} C_{k} r^{2k}, C_0 = C_0^\top \succeq 0, \dots, C_N = C_N^\top \succeq 0, \sum_{k=0}^{N} C_{k} \succeq \varepsilon I,
\label{convex_ISS_d:gamma}\\
&\lambda(x,d)-\varepsilon=s_\lambda(x,d),\label{convex_ISS_d:lambda}\\
&\Theta(x)-\eta\Xi(x)=S_\Theta(x),\\
&\bmat{
\left\{
			\begin{smallmatrix}
			\Tp\Big(
			\smat{
			H(x)P\\
			W(x)Y(x)}^\top \Zbb \frac{\partial \hat{Z}(x)}{\partial x}^\top
			\Big) \\
			+ \Theta(x) + \lambda(x,d) \frac{\partial \hat{Z}(x)}{\partial x} \bar{\mathbf{Q}} 					\frac{\partial \hat{Z}(x)}{\partial x}^\top 
			\end{smallmatrix}
\right\} & \frac{\partial \hat{Z}(x)}{\partial x} & \smat{H(x)P\\W(x)Y(x)}^\top \\[4mm]
\frac{\partial \hat{Z}(x)}{\partial x}^\top & - \Gamma(|d|) & 0\\
\smat{H(x)P\\W(x)Y(x)} & 0 & - \lambda(x,d) \bar{\mathbf{A}}
			} =-S_\partial(x,d). \label{convex_ISS_d:dissipation}
\end{align}
\end{subequations}
Then, if the function
\begin{align}
\label{convex_ISS_d:b_func}
b(\cdot):=\hat{Z}(\cdot)^{\top}P^{-1}\Xi(\cdot)P^{-1}\hat{Z}(\cdot)
\end{align}
is positive definite and radially unbounded,
the closed-loop system
\begin{align*}
\dot{x} = \AB \bmat{ H(x)\hat{Z}(x)\\ W(x)Y(x)P^{-1}\hat{Z}(x)} +d = f^{\tu{p}}_{A,B}(x,d,Y(x)P^{-1}\hat{Z}(x))
\end{align*}
is ISS with respect to the process disturbance $d$ for all $\AB \in \bar{\mathcal{I}}$, and in particular for $\AB[\star]$.
\end{theorem}
\begin{proof}
Let us consider $x \mapsto V(x):=\hat{Z}(x)^\top P^{-1}\hat{Z}(x)$ as an ISS-Lyapunov function.
By the same arguments as in the proof of Theorem~\ref{thm:convex_ISS_w}, there exist class $\mathcal{K}_\infty$ functions $\alpha_1$ and $\alpha_2$ such that for all $x$
\begin{equation}\label{convex_ISS_d:sandwich_V}
   \alpha_1(|x|) \le V(x) \le \alpha_2(|x|).
\end{equation}
We have
\begin{align*}
\dot{x}= f^{\tu{p}}_{A,B}(x,d,Y(x)P^{-1}\hat{Z}(x)) =\AB \smat{H(x)P\\ W(x)Y(x)} P^{-1}\hat{Z}(x)+ d =: F^{\tu{p}}_{A,B}(x,d).
\end{align*}
The derivative of $V$ along solutions of the closed-loop system is
\begin{align*}
& \langle \nabla V(x), F^{\tu{p}}_{A,B}(x,d)\rangle = 2\hat{Z}(x)^\top P^{-1}\frac{\partial \hat{Z}(x)}{\partial x}
\Big(
\AB
\smat{
H(x)P\\
W(x)Y(x)
}
P^{-1}\hat{Z}(x)
+ d
\Big)\\
& = 
\hat{Z}(x)^\top P^{-1}
\Tp\Big(
\frac{\partial \hat{Z}(x)}{\partial x}\AB 
\smat{
H(x)P\\
W(x)Y(x)}
\Big)
P^{-1}\hat{Z}(x)
+
2 \hat{Z}(x)^\top P^{-1} \frac{\partial \hat{Z}(x)}{\partial x} d\\
& = 
\smat{P^{-1}\hat{Z}(x)\\ d}^\top
\smat{
\Tp\Big(
\smat{
H(x)P\\
W(x)Y(x)}^\top \AB^\top \frac{\partial \hat{Z}(x)}{\partial x}^\top
\Big) & \frac{\partial \hat{Z}(x)}{\partial x} \\
\frac{\partial \hat{Z}(x)}{\partial x}^\top & 0
}
\smat{P^{-1}\hat{Z}(x)\\ d}.
\end{align*}
Define 
\begin{align}
\label{convex_ISS_d:a}
a(x):=\hat{Z}(x)^\top P^{-1}\Theta(x)P^{-1}\hat{Z}(x).
\end{align}
We would like to show that
\begin{align}
& \forall x,d, \AB \in \bar{\mathcal{I}}, \label{convex_ISS_d:diss_ineq_almost}\\
& 0 \ge \langle \nabla V(x), F^{\tu{p}}_{A,B}(x,d)\rangle + a(x) - d^\top \Gamma(|d|) d \notag \\
& = 
\smat{P^{-1}\hat{Z}(x)\\ d}^\top
\smat{
\Tp\Big(
\smat{
H(x)P\\
W(x)Y(x)}^\top \AB^\top \frac{\partial \hat{Z}(x)}{\partial x}^\top
\Big) + \Theta(x) & \frac{\partial \hat{Z}(x)}{\partial x} \\
\frac{\partial \hat{Z}(x)}{\partial x}^\top & - \Gamma(|d|)
}
\smat{P^{-1}\hat{Z}(x)\\ d}. \notag
\end{align}
This holds if
\begin{align*}
& \forall x,d, \AB \in \bar{\mathcal{I}}, \quad  0 \succeq 
\smat{
\Tp\Big(
\smat{
H(x)P\\
W(x)Y(x)}^\top \AB^\top \frac{\partial \hat{Z}(x)}{\partial x}^\top
\Big) + \Theta(x) & \frac{\partial \hat{Z}(x)}{\partial x} \\
\frac{\partial \hat{Z}(x)}{\partial x}^\top & - \Gamma(|d|)
}.
\end{align*}
By the equivalent parametrization of the set $\bar{\mathcal{I}}$ in~\eqref{set_overapproximation}, this holds if
\begin{align*}
& \forall x,d, \Upsilon  \text{ with } \| \Upsilon \| \le 1, \\
& 0 \succeq \!
\smat{
\Tp\Big(
\smat{
H(x)P\\
W(x)Y(x)}^\top \Zbb \frac{\partial \hat{Z}(x)}{\partial x}^\top
\Big) + \Theta(x) & \frac{\partial \hat{Z}(x)}{\partial x} \\
\frac{\partial \hat{Z}(x)}{\partial x}^\top & - \Gamma(|d|)
}\!+\!
\smat{
\Tp\Big(
\smat{
H(x)P\\
W(x)Y(x)}^\top \bar{\mathbf{A}}^{-\h} \Upsilon \bar{\mathbf{Q}}^\h \frac{\partial \hat{Z}(x)}{\partial x}^\top
\Big)  & \star  \\
0 & 0
}\\
& \! = \! \smat{
\!\Tp\Big(
\smat{
H(x)P\\
W(x)Y(x)}^\top \Zbb \frac{\partial \hat{Z}(x)}{\partial x}^\top\!
\Big) + \Theta(x) & \frac{\partial \hat{Z}(x)}{\partial x} \\
\frac{\partial \hat{Z}(x)}{\partial x}^\top & - \Gamma(|d|)
}\!\!+\!
\Tp\!
\left(
\smat{
\!\smat{H(x)P\\W(x)Y(x)}^\top \bar{\mathbf{A}}^{-\h}\\
0
}
\!\Upsilon\!
\smat{\bar{\mathbf{Q}}^\h \frac{\partial \hat{Z}(x)}{\partial x}^\top & 0}\!
\right)\!.
\end{align*}
Since $\varepsilon>0$ and $s_\lambda$ is an SOS polynomial, \eqref{convex_ISS_d:lambda} implies that $\lambda(x,d)>0$ for all $x$ and $d$. 
Then, the previous condition holds if
\begin{align*}
& \forall x,d, \Upsilon  \text{ with } \| \Upsilon \| \le 1, \\
& \!\smat{
\Tp\Big(
\smat{
H(x)P\\
W(x)Y(x)}^\top \Zbb \frac{\partial \hat{Z}(x)}{\partial x}^\top
\Big) + \Theta(x) & \frac{\partial \hat{Z}(x)}{\partial x} \\
\frac{\partial \hat{Z}(x)}{\partial x}^\top & -\Gamma(|d|)
}\!+\!
\Tp
\left(
\smat{
\smat{H(x)P\\W(x)Y(x)}^\top \bar{\mathbf{A}}^{-\h}\\
0
}\!
\Upsilon
\smat{\bar{\mathbf{Q}}^\h \frac{\partial \hat{Z}(x)}{\partial x}^\top & 0}
\right) \\
& \!\preceq 
\smat{
\Tp\Big(
\smat{
H(x)P\\
W(x)Y(x)}^\top \Zbb \frac{\partial \hat{Z}(x)}{\partial x}^\top
\Big) + \Theta(x) & \frac{\partial \hat{Z}(x)}{\partial x}\\
\frac{\partial \hat{Z}(x)}{\partial x}^\top & -\Gamma(|d|)
} \\
& \!+\frac{1}{\lambda(x,d)}
\smat{
\smat{H(x)P\\W(x)Y(x)}^\top \\
0
}
\bar{\mathbf{A}}^{-1}
\smat{
\smat{H(x)P\\W(x)Y(x)}^\top \\
0
}^\top
+ \lambda(x,d)
\smat{
\frac{\partial \hat{Z}(x)}{\partial x}\\
0}
\bar{\mathbf{Q}}
\smat{
\frac{\partial \hat{Z}(x)}{\partial x}\\
0}^\top \preceq 0
\end{align*}
where the first inequality follows by completing the square \cite[Lemma~2]{chen2023data} and $\| \Upsilon \| \le 1$.
By Schur complement \cite[p.~28]{boyd1994linear}, the previous condition holds if
\begin{align*}
& \forall x,d,\quad  
\smat{
\left\{
\begin{smallmatrix}
\Tp\Big(
\smat{
H(x)P\\
W(x)Y(x)}^\top \Zbb \frac{\partial \hat{Z}(x)}{\partial x}^\top
\Big) \\
+ \Theta(x) + \lambda(x,d) \frac{\partial \hat{Z}(x)}{\partial x} \bar{\mathbf{Q}} \frac{\partial \hat{Z}(x)}{\partial x}^\top 
\end{smallmatrix}
\right\} & \frac{\partial \hat{Z}(x)}{\partial x} & \smat{H(x)P\\W(x)Y(x)}^\top \\[.5mm]
\frac{\partial \hat{Z}(x)}{\partial x}^\top & - \Gamma(|d|) & 0\\
\smat{H(x)P\\W(x)Y(x)} & 0 & - \lambda(x,d) \bar{\mathbf{A}}
} \preceq 0,
\end{align*}
and this is implied by~\eqref{convex_ISS_d:dissipation}.
We have then shown that \eqref{convex_ISS_d:diss_ineq_almost} holds.

Consider the function $a$ defined in~\eqref{convex_ISS_w:a}.
By the same arguments as in the proof of Theorem~\ref{thm:convex_ISS_w}, there exist a class $\mathcal{K}_\infty$ function $\alpha_3$ such that $\alpha_3(|x|) \le a(x)$ for all $x$ and a class $\mathcal{K}_\infty$ function $\alpha_4$ such that $d^\top \Gamma(|d|) d \le \alpha_4(|d|)$ for all $d$.
We then have, from~\eqref{convex_ISS_d:diss_ineq_almost}, that
\begin{align}
\label{convex_ISS_d:diss_ineq}
& \forall x,d, \AB \in \bar{\mathcal{I}}, \\
& \langle \nabla V(x), F^{\tu{p}}_{A,B}(x,d) \rangle \le -a(x) + d^\top \Gamma(|d|) d \le -\alpha_3(|x|) + \alpha_4(|d|).  \notag
\end{align}
Therefore, we conclude by~\eqref{convex_ISS_d:sandwich_V}, \eqref{convex_ISS_d:diss_ineq} and Definition~\ref{def:ISS_lyap_fun} that for all $\AB \in \bar{\mathcal{I}}$, $V$ is an ISS-Lyapunov function for $\dot{x} = F^{\tu{p}}_{A,B}(x,d)$ and by Fact~\ref{fact:equiv_iss} that for all $\AB \in \bar{\mathcal{I}}$, and in particular for $\AB[\star] \in \bar{\mathcal{I}}$, the system $\dot{x} = F^{\tu{p}}_{A,B}(x,d)$ is input-to-state stable with respect to $d$.
\end{proof}

Theorem~\ref{thm:convex_ISS_d} solves Problem~\ref{problem:process_noise}, whereby $\mathcal{I} \subseteq \bar{\mathcal{I}}$ by Fact~\ref{fact:overapp_set} and the comparison functions of Problem~\ref{problem:actuator_disturbance} can be obtained as shown in the proof of Theorem~\ref{thm:convex_ISS_w}.
Notably, \eqref{convex_ISS_d} does not include any product of decision variables and is thus convex, unlike the biconvex \eqref{biconvex_ISS_d}.
To apply Theorem~\ref{thm:convex_ISS_d}, analogous considerations hold to those reported after Theorem~\ref{thm:convex_ISS_w}, especially for the design parameter $\Xi$.

\subsection{Model-based specialization}
\label{sec:modelbased}

In the previous subsections, we have provided data-based conditions for input-to-state stabilization in Theorems~\ref{thm:biconvex_ISS_w}, \ref{thm:biconvex_ISS_d}, \ref{thm:convex_ISS_w}, \ref{thm:convex_ISS_d}.
In this section we would like to observe that these results can be immediately specialized to a model-based setting, if one knew the matrices $\AB[\star]$ of the actual system.
For brevity, we show this only for Theorem~\ref{thm:convex_ISS_w}, of which we obtain the next model-based corollary.

\begin{corollary}
\label{cor:MB_convex_ISS_w}
Let Assumptions~\ref{assumpt:sys} and \ref{assumpt:Zhat} hold. 
For given symmetric polynomial matrix $\Xi$ and scalar $\varepsilon >0$, suppose there exist a matrix $P=P^\top \succ 0$, polynomial matrices $Y$ and $\Gamma$, symmetric polynomial matrix $\Theta$, a scalar $\eta>0$, SOS polynomial matrices $S_\Theta$ and $S_\partial$ such that for all $r$, $x$ and $w$,
\begin{subequations}%
\label{MB_convex_ISS_w}
\begin{align}
&\Gamma(r)=\sum_{k=0}^{N} C_{k} r^{2k}, C_0 = C_0^\top \succeq 0, \dots, C_N = C_N^\top \succeq 0, \sum_{k=0}^{N} C_{k} \succeq \varepsilon I,
\label{MB_convex_ISS_w:gamma}\\
&\Theta(x)-\eta\Xi(x)=S_\Theta(x),\label{MB_convex_ISS_w:Theta}\\
&
\bmat{
\Tp\Big(\frac{\partial \hat{Z}(x)}{\partial x} \AB[\star] \smat{H(x)P\\W(x)Y(x)} \Big) +\Theta(x) &\star \\[3mm]
\smat{0\\W(x)}^\top \AB[\star]^\top \frac{\partial \hat{Z}(x)}{\partial x}^\top &-\Gamma(|w|)
} = -S_\partial(x,w).\label{MB_convex_ISS_w:dissipation}
\end{align}
\end{subequations}
Then, if the function
$b(\cdot):=\hat{Z}(\cdot)^{\top}P^{-1}\Xi(\cdot)P^{-1}\hat{Z}(\cdot)$
is positive definite and radially unbounded, 
the closed-loop system
\begin{align*}
\dot{x}= \AB[\star] \bmat{H(x)\hat{Z}(x)\\ W(x)\big( Y(x)P^{-1}\hat{Z}(x) +w \big)} 
\end{align*}
is ISS with respect to the actuator disturbance $w$.
\end{corollary}

Corollary~\ref{cor:MB_convex_ISS_w} corresponds to a convex program for the numerical construction of a controller and an ISS-Lyapunov function that guarantee input-to-state stabilization of the closed-loop system for a nonlinear input-affine system with polynomial dynamics.

\section{Numerical illustration}
\label{sec:exper}

To show how the proposed approaches allow for data-driven design of controllers achieving input-to-state stabilization with respect to exogenous inputs, we consider the polynomial system
\begin{equation}\label{simu:sys}
\begin{bmatrix}
\dot{x}_1 \\
\dot{x}_2
\end{bmatrix}=\begin{bmatrix}
-x_1^3+x_1x_2^2 \\
x_1x_2^2-x_1^2x_2
\end{bmatrix}+\begin{bmatrix}
0 \\
1
\end{bmatrix} u=f_{\star}(x)+g_{\star}(x) u
\end{equation}
from \cite{anta2008self}. 
Along the lines of Remark~\ref{rmk:regressor}, we select the function libraries $x \mapsto Z(x)=(x_1^3, x_1^2x_2,x_1x_2^2,x_2^3)$ and $x \mapsto W(x)=1$, which satisfy Assumption~\ref{assumpt:sys}. 
Note that $Z$ contains more monomials than those appearing in $f_{\star}$. 
With these $Z$ and $W$, \eqref{simu:sys} yields
$
A_{\star}=
\smat{
-1 & 0 & 1 & 0\\
0 & -1 & 1 & 0
}$ and $B_{\star}=\smat{
0\\1
}$.
The matrices $A_\star$ and $B_\star$ are unknown in our setting, see Section~\ref{sec:problem}, and are used only to generate the noisy data points on which the programs \eqref{biconvex_ISS_w}, \eqref{biconvex_ISS_d}, \eqref{convex_ISS_w} and \eqref{convex_ISS_d} are based. 
In the numerical experiment, the initial state is $x(0)=(2,-2)$ and the input $u$  is generated by linearly interpolating a sequence that is a realization of a Gaussian random variable with mean zero and variance one.
Analogously, the disturbance $d$ acting on data, see \eqref{data_collect}, is generated by linearly interpolating a sequence that is a realization of a random variable uniformly distributed in $\{ d \in \real^2 \colon |d| \le 1 \}$, so to satisfy Assumption~\ref{assumpt:noise}.
The evolutions of $x$, $u$ and the unknown $d$ are depicted in Figure~\ref{fig:data}. 

\begin{figure}[htbp]
\centerline{\includegraphics[width=0.6\linewidth]{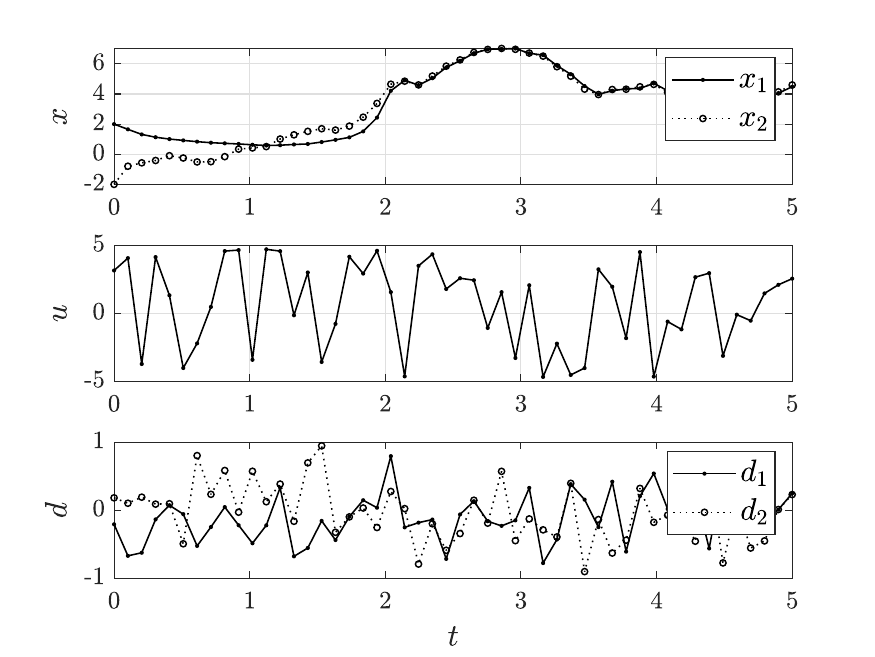}}
\caption{Data-collection experiment for the considered polynomial system.}\label{fig:data}
\end{figure}

From these signals, we collect the data points $\{ u(t_i), x(t_i), \dot{x}(t_i)\}_{i=0}^{T-1}$ for $T=50$ and, with them, we construct the matrices $\boldsymbol{C}_i, \boldsymbol{B}_i$, $\boldsymbol{A}_i$ in \eqref{sol_overapp}, $i=0, \dots, T-1$, solve \eqref{overapp} by YALMIP \cite{lofberg2004yalmip} and obtain matrices $\bar{\mathbf{A}}$, $\bar{\mathbf{B}}$ and, thus, $\Zbb$ and $\bar{\mathbf{Q}}$ in \eqref{solution_overapp}.
The so-obtained matrices $\bar{\mathbf{A}}$, $\bar{\mathbf{B}}$, $\Zbb$ and $\bar{\mathbf{Q}}$ from this single numerical experiment are used in all subsequent programs.

\subsection{Establishing ISS with biconvex programs}
\label{sec:pro_nonconvex}

Using $\bar{\mathbf{A}}$, $\Zbb$ and $\bar{\mathbf{Q}}$, we apply Theorems~\ref{thm:biconvex_ISS_w} or \ref{thm:biconvex_ISS_d} to design a state-feedback controller that achieves input-to-state stabilization with respect to actuator or process disturbances. 
Specifically, we employ YALMIP \cite{lofberg2009pre} to solve \eqref{biconvex_ISS_w} for actuator disturbances or \eqref{biconvex_ISS_d} for process disturbances, setting the allowed maximum degrees for $\lambda$, $k$, $V$ at $4$, $3$, $4$ and the allowed minimum degrees at $0$, $1$, $2$, respectively. 
As elucidated in Section~\ref{subsec:process_noise}, we solve \eqref{biconvex_ISS_w} or \eqref{biconvex_ISS_d} alternately, first with $V$, $\lambda$ (but not $k$) among the decision variables and then with $k$ (but not $V$, $\lambda$) among the decision variables.
For \eqref{biconvex_ISS_w} or \eqref{biconvex_ISS_d}, these two steps are repeated 3 times and require an initial guess for the controller. 
As an initial guess, we take $x \mapsto k_0(x)= -x_2^3-x_1x_2^2$ from \cite{anta2008self}, which ensures global asymptotic stability (of the origin) when $w=0$ or $d=0$; however, it is well known that global asymptotic stability with $w=0$ or $d=0$ does not imply input-to-state stability with respect to $w$ or $d$ \cite{sontag2008input}.

\begin{table}
\resizebox{\linewidth}{!}{
\begin{tabular}{ll}
\toprule
ISS w.r.t. $w$ via \eqref{biconvex_ISS_w} & ISS w.r.t. $d$ via \eqref{biconvex_ISS_d}\\
\midrule
$\begin{matrix}
k^{\tu{b,a}}(x) = 0.2660x_1-1.9477x_2+0.2335x_1^3 \\-0.6436x_1^2x_2 +0.2966x_1x_2^2 -1.0875x_2^3
\end{matrix}$ & 
$\begin{matrix}
k^{\tu{b,p}}(x) = 0.3792x_1-1.0342x_2+0.1849x_1^3 \\ -0.4296x_1^2x_2 -0.2919x_1x_2^2-1.2141x_2^3
\end{matrix}$\\
$\begin{matrix}
V^{\tu{b,a}}(x) =
0.9381x_2^2+0.0267x_1^2-0.2511x_1x_2\\ +2.6152x_1^4 -2.6978x_1^3x_2+3.0112x_1^2x_2^2\\-0.2369x_1x_2^3+4.1744x_2^4
\end{matrix}$ & 
$\begin{matrix}
V^{\tu{b,p}}(x) =
0.8325x_2^2+0.1142x_1^2-0.6086x_1x_2\\+1.5527x_1^4 -2.7115x_1^3x_2+2.2416x_1^2x_2^2\\-0.4772x_1x_2^3+2.9887x_2^4
\end{matrix}$\\
$\begin{matrix}
\alpha_1^{\tu{b,a}}(r) = 0.0049r^2+0.8761r^4
\end{matrix}$ & 
$\begin{matrix}
\alpha_1^{\tu{b,p}}(r) = 0.0013r^2+0.3183r^4
\end{matrix}$\\
$\begin{matrix}
\alpha_2^{\tu{b,a}}(r) = 1.9143r^2+ 5.1702r^4
\end{matrix}$ & 
$\begin{matrix}
\alpha_2^{\tu{b,p}}(r) = 1.3096r^2+3.3943r^4
\end{matrix}$\\
$\begin{matrix}
\alpha_3^{\tu{b,a}}(r) = 0.0077r^4
\end{matrix}$ & 
$\begin{matrix}
\alpha_3^{\tu{b,p}}(r) = 0.0030r^4
\end{matrix}$\\
$\begin{matrix}
\alpha_4^{\tu{b,a}}(r) = 1.9330r^2+8.2237r^4
\end{matrix}$ & 
$\begin{matrix}
\alpha_4^{\tu{b,p}}(r) = 4.1492r^2+1.1602r^4
\end{matrix}$\\
$\begin{matrix}
\lambda^{\tu{b,a}}(x,w) = 9.7939+0.8059x_2^2+0.7872x_1^2 \\-0.0111 x_1 x_2 +0.0001 x_2 w+0.0059 w^2
\end{matrix}$ & 
$\begin{matrix}
\lambda^{\tu{b,p}}(x,d) = 6.6912+0.4736x_2^2+0.4104x_1^2 \hspace*{15mm}\\ -0.08655x_1x_2 +0.0003x_1 d +0.0002 x_2 d +0.0032 d^2
\end{matrix}$\\
\bottomrule
\end{tabular}
}
\caption{Solutions of programs \eqref{biconvex_ISS_w} and \eqref{biconvex_ISS_d}.}
\label{tab:sol}
\label{}
\end{table}

We report the obtained solutions in Table \ref{tab:sol}, where the superscripts $\tu{b}$, $\tu{a}$, $\tu{p}$ stand for ``biconvex'', ``actuator'', ``process''; moreover, for each quantity, we report only the most significant terms and neglect those with much smaller coefficients.
To illustrate these solutions, we use the numerically obtained controllers, ISS-Lyapunov functions and comparison functions to show through Figure~\ref{fig_iss_V} that, for all $t \ge 0$,
\begin{align*}
\dot{V}^{\tu{b,a}}(t):=\langle\nabla V^{\tu{b,a}}(x(t)),f_{A_\star,B_\star}^{\tu{a}}(x(t),w(t),k^{\tu{b,a}}(x(t)))\rangle 
& \le -\alpha_3^{\tu{b,a}}(|x(t)|)+\alpha_4^{\tu{b,a}}(|w(t)|) \\
\dot{V}^{\tu{b,p}}(t):= \langle\nabla V^{\tu{b,p}}(x(t)),f_{A_\star,B_\star}^{\tu{p}}(x(t),d(t),k^{\tu{b,p}}(x(t)))
\rangle 
& \le -\alpha_3^{\tu{b,p}}(|x(t)|)+\alpha_4^{\tu{b,p}}(|d(t)|),
\end{align*}
which is the key requirement of an ISS-Lyapunov function, cf.~\eqref{ISS_lyap_fun_prop:dissip_ineq}.

\begin{figure}[htbp]
\centerline{
\includegraphics[width=0.55\linewidth]{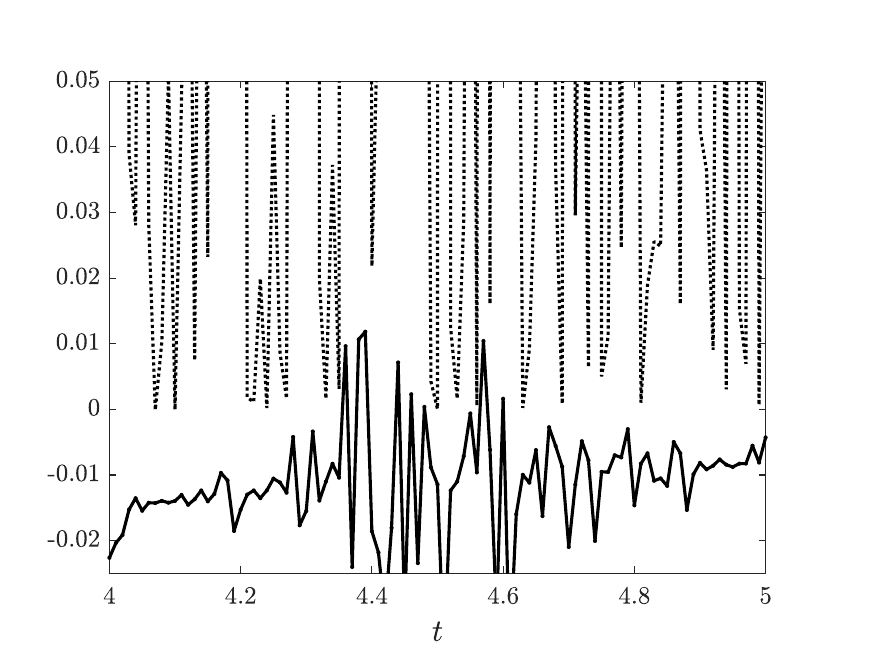}\hspace*{-6mm}
\includegraphics[width=0.55\linewidth]{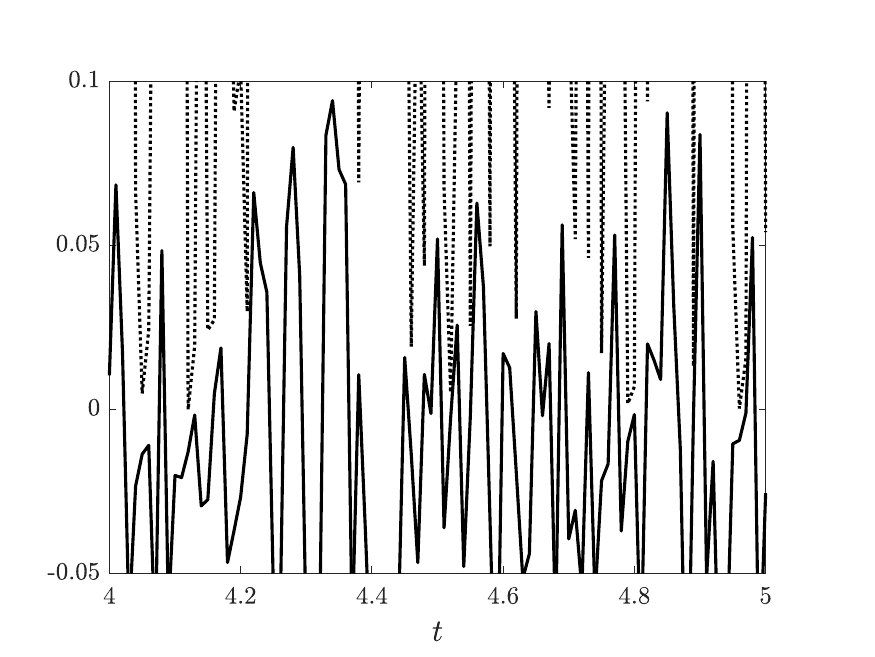}}
\caption{Left: evolution of $\dot{V}^{\tu{b,a}}(\cdot)$ (solid) and $-\alpha_3^{\tu{b,a}}(|x(\cdot)|)+\alpha_4^{\tu{b,a}}(|w(\cdot)|)$ (dotted). Right: evolution of $\dot{V}^{\tu{b,p}}(\cdot)$ (solid) and $-\alpha_3^{\tu{b,p}}(|x(\cdot)|)+\alpha_4^{\tu{b,p}}(|d(\cdot)|)$ (dotted).}
\label{fig_iss_V}
\end{figure}

\subsection{Establishing ISS with convex programs}
\label{sec:pro_convex}

First, we apply Theorem~\ref{thm:convex_ISS_w} to achieve input-to-state stabilization with respect to actuator disturbances. 
We select the function library $x \mapsto \hat{Z}^{\tu{a}}(x) = (x_1,x_2)$, which returns
\begin{align*}
Z(x) = H^{\tu{a}}(x)\hat{Z}^{\tu{a}}(x) =
\smat{x_1^2 & 0\\
x_1x_2 & 0\\
0 & x_1x_2\\
0 & x_2^2
}
\bmat{x_1\\x_2}.
\end{align*} 
For $\hat{Z}^\tu{a}$, Assumption~\ref{assumpt:Zhat} holds. 
With $\bar{\mathbf{A}}$, $\Zbb$ and $\bar{\mathbf{Q}}$, we solve \eqref{convex_ISS_w} with YALMIP taking the allowed maximum degrees of $\lambda$, $Y$ as $4$, $2$, respectively.
To reduce computational complexity, we consider $x \mapsto \lambda(x)$ instead of $(x,w) \mapsto \lambda(x,w)$.
We set $\Xi^{\tu{a}}(x)=\smat{x_1^2&x_1x_2\\x_1x_2&x_2^2}$.
The solutions of \eqref{convex_ISS_w} are
    \begin{align*}
        P^{\tu{a}}&=\smat{
        0.8557 &-0.0226\\-0.0226&0.6136},\\
        Y^{\tu{a}}(x)&=\smat{
        0.0396+0.0462x_1^2+0.1053x_1x_2+0.0313x_2^2 & ~-3.1085-3.2126x_1^2-0.3021x_1x_2-3.21744x_2^2},\\
        \Gamma^{\tu{a}}(r)&= 2.7159+2.4804r^2,\\        
        \Theta^{\tu{a}}(x) & = \smat{
        0.0022+0.5519x_1^2+0.0223x_1x_2+0.0096x_2^2 & -0.0142-0.04919x_1^2+0.2074x_1x_2-0.0186x_2^2\\
        \star & 1.0509+1.4427x_1^2+0.1476x_1x_2+1.4316x_2^2}\\
        \lambda^{\tu{a}}(x)&= 0.0244+0.04898x_1^2-0.0041x_1x_2+0.0458x_2^2+0.0254x_1^4-0.0040x_1^3x_2\\
        & \quad +0.0467x_1^2x_2^2-0.0037x_1x_2^3+0.0224x_2^4\\
        \eta^{\tu{a}} &=0.2104.        
    \end{align*}
From~\eqref{convex_ISS_w:a} and \eqref{convex_ISS_w:b_func}, we obtain
\begin{align*}
a^{\tu{a}}(x)&=0.0035x_1^2+0.0935x_1x_2+2.7942x_2^2+0.7529x_1^4\\
&\quad +0.1220x_1^3x_2+4.6620x_1^2x_2^2+0.5529x_1x_2^3+3.8066x_2^4 \\
b^{\tu{a}}(x)&=1.3685x_1^4+0.2015x_1^3x_2+3.8238x_1^2x_2^2+0.2809x_1x_2^3+2.6608x_2^4.
\end{align*}

Second, we apply Theorem~\ref{thm:convex_ISS_d} to achieve input-to-state stabilization with respect to process disturbances. 
We select the function library $x \mapsto \hat{Z}^{\tu{p}}(x) = (x_1^2,x_2^2)$, which returns
\begin{align*}
Z(x) = H^{\tu{p}}(x) \hat{Z}^{\tu{p}}(x) =
\smat{
x_1& 0\\
x_2& 0\\
0& x_1\\
0& x_2}
\bmat{x_1^2\\
x_2^2}.
\end{align*} 
For $\hat{Z}^{\tu{p}}$, Assumption~\ref{assumpt:Zhat} holds. With $\bar{\mathbf{A}}$, $\Zbb$ and $\bar{\mathbf{Q}}$, we solve \eqref{convex_ISS_d} with YALMIP taking the allowed maximum degrees of $\lambda$, $Y$ as $4$, $2$, respectively. 
To reduce computational complexity, we take $x \mapsto \lambda(x)$ instead of $(x,d) \mapsto \lambda(x,d)$ and $r\mapsto \Gamma(r)=c_1 I_2$ with $c_1>0$.
We set $\Xi^{\tu{p}}(x)=\smat{x_1^2&0\\0&x_2^2}$. 
The solutions of \eqref{convex_ISS_d} are
\begin{align*}
    P^{\tu{p}}&=\smat{
    0.2779&0.0098\\
    0.0098&0.0275},\\
    Y^{\tu{p}}(x)&=\smat{
    0.0766x_1+0.2042x_2 & -0.0387x_1-6.0269x_2},\\
    \Gamma^{\tu{p}}(r)&= \smat{51940.8776&0\\0&51940.8776},\\    
    \Theta^{\tu{p}}(x) & = \smat{
    0.0008+0.0893x_1^2+0.0046x_2^2+0.0151x_1x_2 & 0.0004-0.0041x_1^2+0.01305x_2^2-0.0186x_1x_2 \\
    \star & 0.0049+0.0168x_1^2+1.3703x_2^2+0.0057x_1x_2},\\
    \lambda^{\tu{p}}(x)&=0.1475+0.0170x_1^2+0.0371x_2^2-0.0007x_1x_2\\
    & \quad +0.0012x_1^4+0.0021x_1^2x_2^2+0.0026x_2^4,\\
    \eta^{\tu{p}}&= 0.0394.
\end{align*}
From~\eqref{convex_ISS_d:a} and \eqref{convex_ISS_d:b_func}, we obtain
\begin{align*}
a^{\tu{p}}(x)& =0.0161x_1^4-0.3792x_1^2x_2^2+6.5852x_2^4+1.2548x_1^6-1.3172x_1^4x_2^2+0.3879x_1^5x_2\\
 &\quad-5.7523x_1^3x_2^3-104.6085x_1^2x_2^4+9.5674x_1x_2^5+1852.5509x_2^6\\
b^{\tu{p}}(x) &=13.2875x_1^6-7.8005x_1^4x_2^2-94.1367x_1^2x_2^4+1352.8445x_2^6
.   
\end{align*}

Based on the so-obtained solutions for the cases of actuator and process disturbances, we compute the corresponding controllers, ISS-Lyapunov functions and comparison functions to contrast them with those found in Section~\ref{sec:pro_nonconvex} for the same polynomial system.
Controller and ISS-Lyapunov function are computed based on their definitions as
\begin{align*}
k^{\tu{c},\ell}(x) = Y^\ell(x) (P^\ell)^{-1} \hat{Z}^\ell(x), \quad V^{\tu{c},\ell}(x) = \hat{Z}^\ell(x)^\top (P^\ell)^{-1} \hat{Z}^\ell(x), \quad \ell \in \{\tu{a}, \tu{p}\}
\end{align*}
where the superscripts $\tu{c}$, $\tu{a}$, $\tu{p}$ stand for ``convex'', ``actuator'', ``process''.
The comparison functions $\alpha^{\tu{c},\ell}_1$, $\alpha^{\tu{c},\ell}_2$, $\alpha^{\tu{c},\ell}_3$, $\ell \in \{ \tu{a}, \tu{p} \}$, can be computed from $a^{\tu{a}}$, $b^{\tu{a}}$, $a^{\tu{p}}$, $b^{\tu{p}}$ by using YALMIP to solve, for $\ell \in \{ \tu{a}, \tu{p} \}$,
\begin{align*}
V^{\tu{c},\ell}(x) - \alpha^{\tu{c},\ell}_1(|x|) = s^\ell_{\alpha_1}(x), \,\, \alpha^{\tu{c},\ell}_2(|x|) - V^{\tu{c},\ell}(x) = s^\ell_{\alpha_2}(x), \,\,
a^\ell(x) - \alpha^{\tu{c},\ell}_3(|x|) = s^\ell_{\alpha_3}(x)
\end{align*}
for SOS polynomials $s^\ell_{\alpha_1}$, $s^\ell_{\alpha_2}$, $s^\ell_{\alpha_3}$ and $\alpha^{\tu{c},\ell}_1$, $\alpha^{\tu{c},\ell}_2$, $\alpha^{\tu{c},\ell}_3$ as in Lemma~\ref{lem:classk}.
Finally, we have from the proofs of Theorems~\ref{thm:convex_ISS_w} and \ref{thm:convex_ISS_d} that
\begin{align*}
\alpha^{\tu{c,a}}_4(r)=\sum_{k=1}^{N+1} \lambda_{\max}(C^{\tu{a}}_{k-1}) r^{2k}, \quad \alpha^{\tu{c,p}}_4(r)=\sum_{k=1}^{N+1} \lambda_{\max}(C^{\tu{p}}_{k-1}) r^{2k}. 
\end{align*}
All these quantities are summarized in Table~\ref{tab:solconvex}, which parallels Table~\ref{tab:sol}.
The numerically obtained controllers, ISS-Lyapunov functions and comparison functions satisfy a dissipation inequality like the one in~\eqref{ISS_lyap_fun_prop:dissip_ineq} so that for all $t \ge 0$, 
\begin{align*}
& \dot{V}^{\tu{c,a}}(t):=\langle\nabla V^{\tu{c,a}}(x(t)),f_{A_\star,B_\star}^{\tu{a}}(x(t),w(t),Y^{\tu{a}}(x(t))(P^{\tu{a}})^{-1} \hat{Z}^{\tu{a}}(x(t)))\rangle \\
& \hspace*{8cm}\le -\alpha_3^{\tu{c,a}}(|x(t)|)+\alpha_4^{\tu{c,a}}(|w(t)|) \\
& \dot{V}^{\tu{c,p}}(t):= \langle\nabla V^{\tu{c,p}}(x(t)),f_{A_\star,B_\star}^{\tu{p}}(x(t),d(t),Y^{\tu{p}}(x(t))(P^{\tu{p}})^{-1} \hat{Z}^{\tu{p}}(x(t)))\rangle \\
& \hspace*{8cm}\le -\alpha_3^{\tu{c,p}}(|x(t)|)+\alpha_4^{\tu{c,p}}(|d(t)|).
\end{align*}
This is shown in Figure~\ref{fig:iss_convex}, which parallels Figure~\ref{fig_iss_V}.

\begin{table}[ht!]
\resizebox{\linewidth}{!}{
\begin{tabular}{ll}
\toprule
ISS w.r.t. $w$ via \eqref{convex_ISS_w} & ISS w.r.t. $d$ via \eqref{convex_ISS_d}\\
\midrule
$\begin{matrix}
k^{\tu{c,a}}(x) = -0.0875x_1-5.0689x_2-0.0843x_1^3\\-5.1281x_1^2x_2 -0.5903x_1x_2^2-5.2469x_2^3
\end{matrix}$ & 
$\begin{matrix}
k^{\tu{c,p}}(x) = 0.3295x_1^3+8.5960x_1^2x_2\\-1.5220x_1x_2^2-221.9413x_2^3
\end{matrix}$\\
$\begin{matrix}
V^{\tu{c,a}}(x) =1.1698x_1^2+0.0861x_1x_2+1.6312x_2^2
\end{matrix}$ & 
$\begin{matrix}
V^{\tu{c,p}}(x) =3.6452x_1^4-2.6055x_1^2x_2^2+36.7810x_2^4
\end{matrix}$\\
$\begin{matrix}
\alpha_1^{\tu{c,a}}(r) = 1.1658r^2
\end{matrix}$ & 
$\begin{matrix}
\alpha_1^{\tu{c,p}}(r) = 3.5941r^4
\end{matrix}$\\
$\begin{matrix}
\alpha_2^{\tu{c,a}}(r) =1.6353r^2
\end{matrix}$ & 
$\begin{matrix}
\alpha_2^{\tu{c,p}}(r) = 36.8322r^4
\end{matrix}$\\
$\begin{matrix}
\alpha_3^{\tu{c,a}}(r) = 0.0017r^2+0.3501r^4
\end{matrix}$ & 
$\begin{matrix}
\alpha_3^{\tu{c,p}}(r) = 0.0034r^4+ 0.3483r^6
\end{matrix}$\\
$\begin{matrix}
\alpha_4^{\tu{c,a}}(r) = 2.7159r^2+2.4804r^4
\end{matrix}$ & 
$\begin{matrix}
\alpha_4^{\tu{c,p}}(r) = 51941r^2
\end{matrix}$\\
\bottomrule
\end{tabular}
}
\caption{Solutions of programs \eqref{convex_ISS_w} and \eqref{convex_ISS_d}.}
\label{tab:solconvex}
\end{table}

\begin{figure}[htbp]
\centerline{
\includegraphics[width=0.55\linewidth]{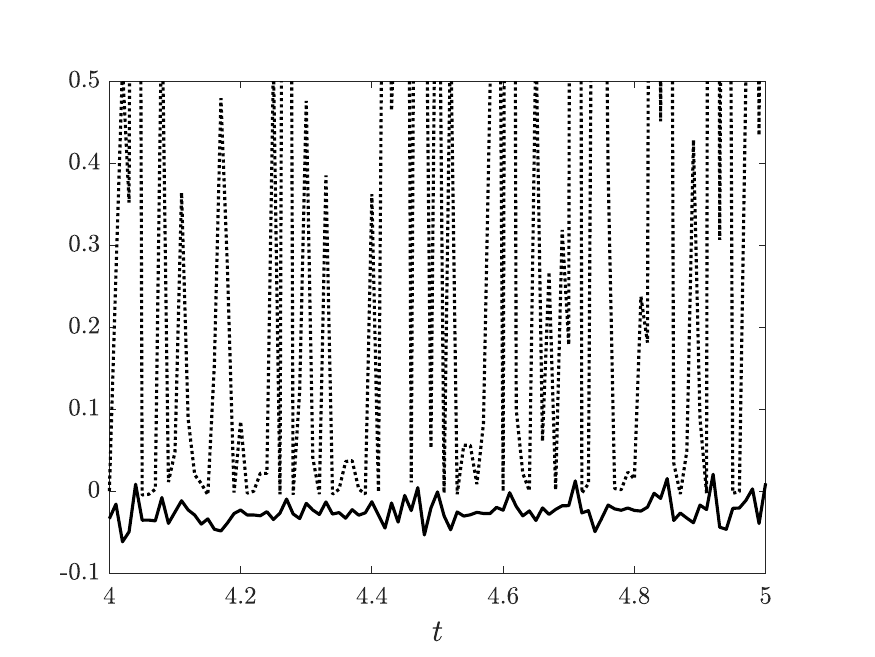}\hspace*{-6mm}
\includegraphics[width=0.55\linewidth]{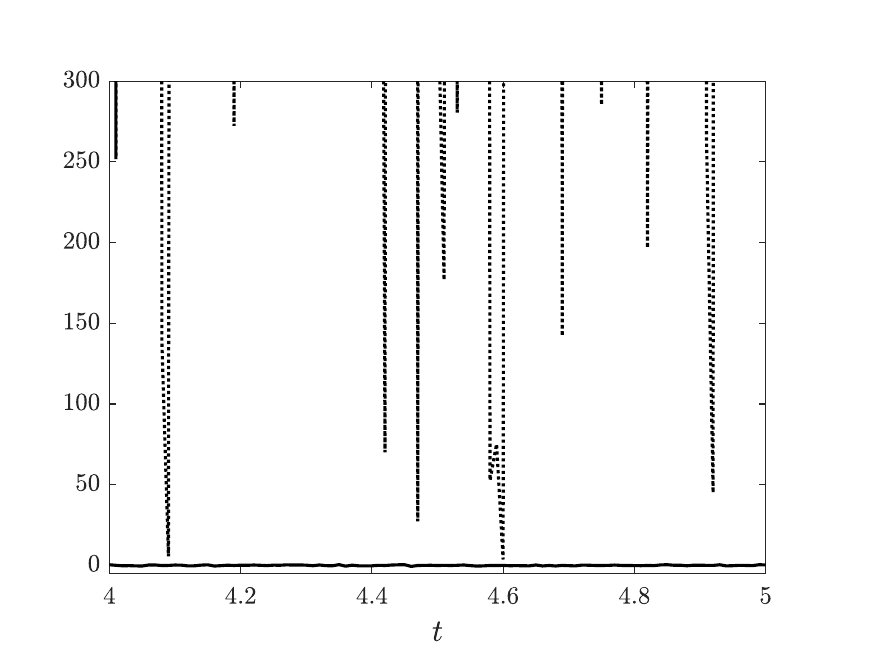}
}
\caption{
Left: evolution of $\dot{V}^{\tu{c,a}}(\cdot)$ (solid) and $-\alpha_3^{\tu{c,a}}(|x(\cdot)|)+\alpha_4^{\tu{c,a}}(|w(\cdot)|)$ (dotted).
Right: evolution of $\dot{V}^{\tu{c,p}}(\cdot)$ (solid) and $-\alpha_3^{\tu{c,p}}(|x(\cdot)|)+\alpha_4^{\tu{c,p}}(|d(\cdot)|)$ (dotted).
}
\label{fig:iss_convex}
\end{figure}

\begin{table}
\begin{tabular}{cccccc}
\toprule
Program & Variables & Scalar  & SOS  & Matrix  & Computation  \\
 &  &  constraints &  constraints &  constraints &  time (s)\\
\midrule
\eqref{biconvex_ISS_w}	& 35 (iter.~1) & 12 (iter.~1) & 4 (iter.~1) & 0 (iter.~1) & 10.7709\\
						& 17 (iter.~2) & 12 (iter.~2) & 3 (iter.~2) & 0 (iter.~2) & (for 6 iter.)\\
\eqref{biconvex_ISS_d}	& 35 (iter.~1) & 12 (iter.~1) & 4 (iter.~1) & 0 (iter.~1) & 9.1363 \\
						& 17 (iter.~2) & 12 (iter.~2) & 3 (iter.~2) & 0 (iter.~2) & (for 6 iter.) \\
\eqref{convex_ISS_w}		& 51 & 4 & 3 & 1 (2-by-2) & 1.7793\\
\eqref{convex_ISS_d}		& 50 & 2 & 3 & 1 (2-by-2) & 1.6937\\
\bottomrule
\end{tabular}
\caption{Main features of the devised SOS programs for computational complexity.}
\label{tab:simu_complexity}
\end{table}

\section{Discussion}
\label{sec:dis}

For the previous results, we discuss computational complexity and compare the biconvex and convex approaches.

\subsection{Computational complexity}

To discuss the computational complexity of the devised SOS programs, we determine
\begin{itemize}[left=0pt]
\item number of decision variables,
\item number of constraints and
\item computation time
\end{itemize}
for each of these SOS programs.
In the number of decision variables, we consider scalars, the free coefficients in matrix variables and the coefficients of polynomials\footnote{
For a polynomial with $n$ variables, minimum degree $\underline{d}$ and maximum degree $\overline{d}$, the number of coefficients is ${n+\overline{d} \choose \overline{d}} - {n+\underline{d}-1 \choose \underline{d}-1}$.}.
For the biconvex programs \eqref{biconvex_ISS_w} and \eqref{biconvex_ISS_d}, each of the two alternate iterations has in principle a different number of decision variables and constraints.
The computation times for solving each program are obtained by the MATLAB\textsuperscript{\textregistered} R2022a function tic toc
on a machine with processor Intel\textsuperscript{\textregistered} Core\textsuperscript{\texttrademark} i7 with 8 cores and 2.50 GHz.
In the case of the biconvex programs \eqref{biconvex_ISS_w} and \eqref{biconvex_ISS_d}, we consider as computation time the time required to perform three times the two iterations, for a total of six SOS programs solved.

The overall results are in Table~\ref{tab:simu_complexity}.
While the convex programs are one-shot, the time taken to solve them is comparable to that needed for solving a single iteration for \eqref{biconvex_ISS_w} or \eqref{biconvex_ISS_d}, namely, \eqref{biconvex_ISS_w} or \eqref{biconvex_ISS_d} that has become convex after fixing some decision variables.
 
\subsection{Comparison between biconvex and convex programs} 

Feasibility of the biconvex and convex programs is only a sufficient condition, rather than necessary and sufficient, for input-to-state stabilization with respect to exogenous inputs.
Further, the next elements are worth mentioning for a comparison between the two approaches.
\begin{enumerate}[left=-3pt,label=\alph*)]
\item In~\eqref{biconvex_ISS_w} and \eqref{biconvex_ISS_d}, the ISS-Lyapunov function is any polynomial whereas in~\eqref{convex_ISS_w} and \eqref{convex_ISS_d} it is parametrized as $\hat{Z}(x)^\top P^{-1}\hat{Z}(x)$ (with $P$ positive definite).
The latter is generally more restrictive.
\item \eqref{biconvex_ISS_w} and \eqref{biconvex_ISS_d} rely on an initial guess to deal iteratively with the products of decision variables, whereas  \eqref{convex_ISS_w} and \eqref{convex_ISS_d} do not.
\item When considering a dissipation inequality of the type \eqref{ISS_lyap_fun_prop:dissip_ineq}, it may be desirable that the function $\alpha_3$ is ``large''  and the function $\alpha_4$ is ``small'' so to have ``large'' decrease through the state and ``small'' increase induced by the exogenous input.
When contrasting such comparison functions for the biconvex and the convex programs in Tables~\ref{tab:sol} and \ref{tab:solconvex}, those are comparable (save for $\alpha_4^{\tu{c,p}}$).
\end{enumerate}

All in all, the two approaches have each their advantages and disadvantages.

\section{Conclusion}
\label{sec:conclu}
Starting from noisy data generated by a nonlinear input-affine system with polynomial dynamics, we have obtained conditions for input-to-state stabilization of all systems consistent with data, with respect to actuator and process disturbances.

Future work is aimed at understanding how the proposed results can serve towards data-driven input-to-state stabilization of input-affine nonlinear systems with nonpolynomial dynamics, including how more general parametrizations of comparison functions could make the corresponding bounds tighter.

\bibliographystyle{siamplain}
\bibliography{references}

\begin{thebibliography}{10}

\bibitem{ahmadi2023safely}
{\sc A.~A. Ahmadi, A.~Chaudhry, V.~Sindhwani, and S.~Tu}, {\em Safely learning
  dynamical systems}, arXiv preprint arXiv:2305.12284,  (2023).

\bibitem{ahmadi2023learning}
{\sc A.~A. Ahmadi and B.~E. Khadir}, {\em Learning dynamical systems with side
  information}, SIAM Review, 65 (2023), pp.~183--223.

\bibitem{ahmadi2011globally}
{\sc A.~A. Ahmadi, M.~Krstic, and P.~A. Parrilo}, {\em A globally
  asymptotically stable polynomial vector field with no polynomial {L}yapunov
  function}, in Proc. IEEE Conference on Decision and Control and European
  Control Conference, 2011, pp.~7579--7580.

\bibitem{anta2008self}
{\sc A.~Anta and P.~Tabuada}, {\em Self-triggered stabilization of homogeneous
  control systems}, in American Control Conference, 2008, pp.~4129--4134.

\bibitem{bisoffi2021trade}
{\sc A.~Bisoffi, C.~De~Persis, and P.~Tesi}, {\em Trade-offs in learning
  controllers from noisy data}, Systems \& Control Letters, 154 (2021).

\bibitem{bisoffi2022data}
{\sc A.~Bisoffi, C.~De~Persis, and P.~Tesi}, {\em Data-driven control via
  {P}etersen's lemma}, Automatica, 145 (2022).

\bibitem{boyd1994linear}
{\sc S.~Boyd, L.~El~Ghaoui, E.~Feron, and V.~Balakrishnan}, {\em Linear matrix
  inequalities in system and control theory}, SIAM, 1994.

\bibitem{brunton2016discovering}
{\sc S.~L. Brunton, J.~L. Proctor, and J.~N. Kutz}, {\em Discovering governing
  equations from data by sparse identification of nonlinear dynamical systems},
  Proceedings of the National Academy of Sciences, 113 (2016), pp.~3932--3937.

\bibitem{chen2023data}
{\sc H.~Chen, A.~Bisoffi, and C.~De~Persis}, {\em Data-driven input-to-state
  stabilization with respect to measurement errors}, in Proc. IEEE Conference
  on Decision and Control, 2023, pp.~1601--1606.

\bibitem{chesi2010lmi}
{\sc G.~Chesi}, {\em {LMI} techniques for optimization over polynomials in
  control: a survey}, IEEE Transactions on Automatic Control, 55 (2010),
  pp.~2500--2510.

\bibitem{de2023event}
{\sc C.~De~Persis, R.~Postoyan, and P.~Tesi}, {\em Event-triggered control from
  data}, IEEE Transactions on Automatic Control,  (2023).

\bibitem{de2023learning}
{\sc C.~De~Persis and P.~Tesi}, {\em Learning controllers for nonlinear systems
  from data}, Annual Reviews in Control, 100915 (2023).

\bibitem{garnier2003continuous}
{\sc H.~Garnier, M.~Mensler, and A.~Richard}, {\em Continuous-time model
  identification from sampled data: implementation issues and performance
  evaluation}, International Journal of Control, 76 (2003), pp.~1337--1357.

\bibitem{gorski2007biconvex}
{\sc J.~Gorski, F.~Pfeuffer, and K.~Klamroth}, {\em Biconvex sets and
  optimization with biconvex functions: a survey and extensions}, Mathematical
  methods of operations research, 66 (2007), pp.~373--407.

\bibitem{grune2023examples}
{\sc L.~Gr{\"u}ne and M.~Sperl}, {\em Examples for separable control {L}yapunov
  functions and their neural network approximation}, IFAC-PapersOnLine, 56
  (2023), pp.~19--24.

\bibitem{guo2021data}
{\sc M.~Guo, C.~De~Persis, and P.~Tesi}, {\em Data-driven stabilization of
  nonlinear polynomial systems with noisy data}, IEEE Transactions on Automatic
  Control, 67 (2021), pp.~4210--4217.

\bibitem{horn2013matrix}
{\sc R.~A. Horn and C.~R. Johnson}, {\em Matrix analysis Second Edition},
  Cambridge university press, 2013.

\bibitem{ichihara2012sum}
{\sc H.~Ichihara}, {\em Sum of squares based input-to-state stability analysis
  of polynomial nonlinear systems}, SICE Journal of Control, Measurement, and
  System Integration, 5 (2012), pp.~218--225.

\bibitem{isidori1999nonlinear}
{\sc A.~Isidori}, {\em Nonlinear Control Systems {II}}, Springer, 1999.

\bibitem{jarvis2005control}
{\sc Z.~Jarvis-Wloszek, R.~Feeley, W.~Tan, K.~Sun, and A.~Packard}, {\em
  Control applications of sum of squares programming}, in Positive Polynomials
  in Control, Springer, 2005, pp.~3--22.

\bibitem{khalil2002nonlinear}
{\sc H.~K. Khalil}, {\em Nonlinear systems, 3rd ed.}, Prentice Hall, 2002.

\bibitem{krstic1998inverse}
{\sc M.~Krstic and Z.-H. Li}, {\em Inverse optimal design of input-to-state
  stabilizing nonlinear controllers}, IEEE Transactions on Automatic Control,
  43 (1998), pp.~336--350.

\bibitem{lavaei2023data}
{\sc A.~Lavaei and D.~Angeli}, {\em Data-driven stability certificate of
  interconnected homogeneous networks via {ISS} properties}, IEEE Control
  Systems Letters,  (2023).

\bibitem{liberzon2002universal}
{\sc D.~Liberzon, E.~D. Sontag, and Y.~Wang}, {\em Universal construction of
  feedback laws achieving {ISS} and integral-{ISS} disturbance attenuation},
  Systems \& Control Letters, 46 (2002), pp.~111--127.

\bibitem{lofberg2004yalmip}
{\sc J.~L\"{o}fberg}, {\em {YALMIP}: A toolbox for modeling and optimization in
  {MATLAB}}, in Proc. IEEE Int. Symp. Comp. Aid. Contr. Sys. Design, 2004.

\bibitem{lofberg2009pre}
{\sc J.~L\"{o}fberg}, {\em Pre-and post-processing sum-of-squares programs in
  practice}, IEEE Transactions on Automatic Control, 54 (2009), pp.~1007--1011.

\bibitem{luppi2023data}
{\sc A.~Luppi, A.~Bisoffi, C.~De~Persis, and P.~Tesi}, {\em Data-driven design
  of safe control for polynomial systems}, European Journal of Control, 100914
  (2023).

\bibitem{majumdar2013control}
{\sc A.~Majumdar, A.~A. Ahmadi, and R.~Tedrake}, {\em Control design along
  trajectories with sums of squares programming}, in 2013 IEEE International
  Conference on Robotics and Automation, 2013, pp.~4054--4061.

\bibitem{martin2023guarantees}
{\sc T.~Martin, T.~B. Sch{\"o}n, and F.~Allg{\"o}wer}, {\em Guarantees for
  data-driven control of nonlinear systems using semidefinite programming: A
  survey}, Annual Reviews in Control, 100911 (2023).

\bibitem{milanese2004set}
{\sc M.~Milanese and C.~Novara}, {\em Set membership identification of
  nonlinear systems}, Automatica, 40 (2004), pp.~957--975.

\bibitem{Mironchenko2023input}
{\sc A.~Mironchenko}, {\em Input-to-State Stability: Theory and Applications},
  Springer, 2023.

\bibitem{parrilo2003semidefinite}
{\sc P.~A. Parrilo}, {\em Semidefinite programming relaxations for
  semialgebraic problems}, Mathematical programming, 96 (2003), pp.~293--320.

\bibitem{prajna2004nonlinear}
{\sc S.~Prajna, A.~Papachristodoulou, and F.~Wu}, {\em Nonlinear control
  synthesis by sum of squares optimization: A {Lyapunov-based} approach}, in
  Proc. Asian Control Conference, vol.~1, 2004, pp.~157--165.

\bibitem{sontag1989smooth}
{\sc E.~D. Sontag}, {\em Smooth stabilization implies coprime factorization},
  IEEE Transactions on Automatic Control, 34 (1989), pp.~435--443.

\bibitem{sontag2008input}
{\sc E.~D. Sontag}, {\em Input to State Stability: Basic Concepts and Results},
  Springer, 2008, pp.~163--220.

\bibitem{sontag1995characterizations}
{\sc E.~D. Sontag and Y.~Wang}, {\em On characterizations of the input-to-state
  stability property}, Systems \& Control Letters, 24 (1995), pp.~351--359.

\bibitem{tabuada2007event}
{\sc P.~Tabuada}, {\em Event-triggered real-time scheduling of stabilizing
  control tasks}, IEEE Transactions on Automatic Control, 52 (2007),
  pp.~1680--1685.

\end{thebibliography}

\end{document}